\documentclass[11pt]{amsart}

\usepackage{comment}

\usepackage[textwidth=17.5cm,textheight=22.275cm,
 height=24.275cm,width=18cm]{geometry}
\usepackage[utf8]{inputenc}
\usepackage[foot]{amsaddr}
\usepackage[english]{babel}
\usepackage{amssymb,amsthm}
\usepackage{mathrsfs}
\usepackage{mathtools}
\mathtoolsset{showonlyrefs,centercolon}
\usepackage[renew-dots,renew-matrix]{nicematrix}
\usepackage{enumerate}
\usepackage{Baskervaldx}
\usepackage{newtxmath}
\usepackage[pdftex,hyperfootnotes,hidelinks]{hyperref}

\allowdisplaybreaks[1]
\theoremstyle{plain}

\newtheorem{teo}{Theorem}[section]
\newtheorem{coro}[teo]{Corollary}

\newtheorem{pro}[teo]{Proposition}

\theoremstyle{definition}

\theoremstyle{remark}
\newtheorem{rem}[teo]{Remark}

\renewcommand{\d}{\operatorname{d}}
\newcommand{\Exp}[1]{\operatorname{e}^{#1}}
\renewcommand{\Re}{\operatorname{Re}}
\newcommand{\N}{\mathbb{N}}
\newcommand{\R}{\mathbb{R}}
\newcommand*\pFqskip{8mu}
\catcode`,\active
\newcommand*\pFq{\begingroup
 \catcode`\,\active
 \def ,{\mskip\pFqskip\relax}%
 \dopFq
}
\catcode`\,12
\def\dopFq#1#2#3#4#5{%
 {}_{#1}F_{#2}\biggl[\genfrac..{0pt}{}{#3}{#4};#5\biggr]%
 \endgroup
}
\title[Pi\~neiro mixed orthogonal polynomials and bidiagonal factorization]{Pi\~neiro mixed orthogonal polynomials and bidiagonal factorization}
\date{July 29, 2026}

\keywords{Mixed multiple orthogonal polynomials, Pi\~neiro systems, hypergeometric series, recurrence matrices, Christoffel transformations, bidiagonal factorization}

\author[Branquinho]{Amílcar Branquinho\( ^1\)}
\address{\( ^1\)CMUC, Departamento de Matem\'atica,
	Universidade de Coimbra, Largo D. Dinis, 3000-143 Coimbra, Portugal}
\email{\(^1\)ajplb@mat.uc.pt}

\author[Díaz]{Juan E. F. Díaz\(^{2}\)}
\address{\(^2\)CIDMA, Departamento de Matemática, Universidade de Aveiro, 3810-193 Aveiro, Portugal}
\email{\(^2\)juan.enri@ua.pt}

\author[Foulquié]{Ana Foulquié-Moreno\(^3\)}
\address{\(^3\)CIDMA, Departamento de Matemática, Universidade de Aveiro, 3810-193 Aveiro, Portugal}
\email{\(^3\)foulquie@ua.pt}

\author[Mañas]{Manuel Mañas\(^4\)}
\address{\(^4\)Departamento de Física Teórica, Universidad Complutense de Madrid, 28040-Madrid, Spain}
\email{\(^4\)manuel.manas@ucm.es}

\thanks{\(^1\)Acknowledges Centre for Mathematics of the University of Coimbra funded
by the Portuguese Government through FCT/MCTES, DOI: 10.54499/UIDB/00324/2020.}

\thanks{\(^2\) and \(^3\) acknowledge the Center for Research and Development
in Mathematics and Applications (CIDMA) from University of Aveiro funded by the Portuguese Foundation
for Science and Technology (FCT) through projects DOI: 10.54499/UIDB/04106/2020 and DOI:
10.54499/UIDP/04106/2020. Additionally, \(^2\) acknowledges PhD contract
DOI: 10.54499/UI/BD/152576/2022 from FCT}

\thanks{\(^2\) and \(^4\) acknowledge research project [PID2021- 122154NB-I00],
\emph{Ortogonalidad y Aproximación con Aplicaciones en Machine Learning y Teoría de la Probabilidad}, funded by \href{https://doi.org/10.13039/501100011033}{MICIU/AEI/10.13039/501100011033} and by \textquotedblleft ERDF A Way of Making Europe\textquotedblright.}

\subjclass{42C05,33C45,33C47}

\begin{document}

\begin{abstract}
We study mixed multiple orthogonal polynomials associated with rank-one
matrices of measures and their recurrence relations. On the step line we
identify the corresponding banded recurrence matrix through the Gauss--Borel
factorization of the moment matrix. Elementary Christoffel transformations
lead to a factorization of this recurrence matrix into lower and upper
bidiagonal factors. For the mixed Pi\~neiro system we derive contour-integral
and generalized hypergeometric representations and formulas for the
recurrence and factorization coefficients.
\end{abstract}

\maketitle


\section{Introduction}

 Mixed multiple orthogonal polynomials \cite{afm,Evi_Arno,sorokin} constitute a natural extension of the classical theory of multiple orthogonal polynomials. They arise in several areas of mathematics and mathematical physics. In probability theory, for example, they appear in the analysis of stochastic models such as Brownian bridges and systems of non-intersecting Brownian motions with \(p\) starting points and \(q \) endpoints \cite{Evi_Arno}. They are also closely connected with multicomponent Toda lattices \cite{adler,afm}. Their relevance extends to number theory, where mixed Hermite-Pad\'e approximation has proved to be a powerful tool. A celebrated example is Ap\'ery's proof of the irrationality of \( \zeta(3) \), which relies on a mixed Hermite-Pad\'e approximation involving three functions \cite{Apery}.
Similar approximation techniques have subsequently been employed to show that infinitely many odd values of the Riemann zeta function are irrational \cite{Ball_Rivoal}, as well as to establish that at least one among \( \zeta(5), \zeta(7), \zeta(9) \), and \( \zeta(11) \) is irrational \cite{Zudilin}.

Bidiagonal factorizations are central tools in the analysis of structured
matrices; see, for example, \cite{Fallat-Johnson,Pinkus book,Contemporary}.
They are also one of the fundamental ingredients in spectral Favard theorems
for banded operators \cite{BFM2,BFM1,BFM1_1}. Beyond this setting,
bidiagonal factorizations appear naturally in Hermite--Pad\'e approximation~\cite{Aptekarev_Kaliaguine_VanIseghem}, integrable systems and Darboux
transformations \cite{dolores_ana_ab1}, and Christoffel transformations
\cite{BFM_22}. From~the probabilistic viewpoint, they correspond to
stochastic factorizations of transition matrices of Markov chains~\cite{BDFM_finite}. Explicit factorizations for recurrence matrices of
several families of multiple orthogonal polynomials have been obtained in
\cite{Aptekarev_Kaliaguine_VanIseghem,BFM_24,BFM_23,Lima-Loureiro}; further
connections with branched continued fractions and combinatorics appear in
\cite{Lima,Sokal,BDFHM}.

Reference~\cite{BFM1} establishes a simultaneous Gaussian quadrature formula
for mixed orthogonality. More recently, the authors of
\cite{Laudadio-Mastronardi-Dooren} developed a method for computing the nodes and weights of simultaneous
n-point Gaussian quadrature rules.
Their approach computes the relevant eigenvalues and eigenvectors in a
numerically stable way.

We study recurrence relations satisfied by mixed orthogonal linear forms and
their type-I and type-II polynomial components. For step-line multi-indices,
we refine the algorithm of \cite{banda_bidiagonal} for the bidiagonal
factorization of the recurrence matrix. In our normalization,
the diagonal entries of the lower bidiagonal factors and the superdiagonal
entries of the upper bidiagonal factors equal \(1\). As a case study, we
consider mixed Pi\~neiro polynomials. We derive contour-integral and
hypergeometric representations for both types of mixed linear forms, using
Mellin-transform techniques related to \cite{BDFMW,VAWolfs,Wolfs}. These
representations yield recurrence coefficients, from which the factorization
coefficients on the step line are obtained exactly.

We begin by recalling the required definitions.

\subsection{Generalized hypergeometric series}

For \(n\in\N_0\), define the Pochhammer symbol algebraically by
\begin{align*}
 (z)_0=1, && (z)_n=\prod_{r=0}^{n-1}(z+r), && n\ge1.
\end{align*}
Whenever the quotient is defined, this agrees with
\((z)_n=\Gamma(z+n)/\Gamma(z)\); elsewhere the algebraic definition, or its
meromorphic continuation after cancellation, is used. For later compact
step-line formulas we also~use
\begin{align}
\label{eq:negative-pochhammer}
 (z)_{-n}:=\frac{1}{(z-n)_n},&& n\in\N,
\end{align}
away from the poles of the right-hand side.

The generalized hypergeometric series is
		\begin{align}
	\label{HF}
	\pFq{p}{q}{a_1,\ldots,a_p}{\alpha_1,\ldots,\alpha_q}{x}\coloneqq \sum_{l=0}^{\infty}\dfrac{(a_1)_l\cdots(a_p)_l}{(\alpha_1)_l\cdots(\alpha_q)_l}\dfrac{x^l}{l!},
	\end{align}
provided that no denominator Pochhammer vanishes before termination. All
hypergeometric series used below terminate because one upper parameter is a
nonpositive integer; thus no additional convergence condition is needed in
the formulas of this paper. Here Euler's gamma function is
\begin{align*}
\Gamma(z)\coloneq\int_0^\infty t^{z-1}\Exp{-t}\d t, && \Re{z}>0.
\end{align*}

\begin{rem}
	If \(n\in\N_0\), then
\begin{align*}
\displaystyle
		(-n)_m=
\begin{cases}
		0,\text{ if } m>n,\\
		(-1)^m\frac{n!}{(n-m)!},\text{ if } m\leqslant n.
\end{cases}
\end{align*}
	This implies that if any of the \(a_1,\ldots,a_p\)
	is a negative integer, then the corresponding generalized hypergeometric series \eqref{HF} is finite. Moreover, if any of the \(a_1,\ldots,a_p\) is zero, then the corresponding function \eqref{HF} equals \(1\).
\end{rem}

\subsection{Mixed orthogonal polynomials}

We consider a matrix of measures
\begin{align*}
	\d	\mu=\begin{bNiceMatrix}
 \d\mu_{1,1}&\Cdots &\d\mu_{1,p}\\
 \Vdots & & \Vdots\\
 \d	\mu_{q,1}&\Cdots &\d\mu_{q,p}
	\end{bNiceMatrix},
\end{align*}
where the measures \(\mu_{i,j}\) are supported on the interval \(\Delta \subseteq \mathbb{R}\). \par
Let \(\vec n=(n_1,\ldots,n_p)\) and
\(\vec m=(m_1,\ldots,m_q)\) be multi-indices, and write
\(|\vec n|=n_1+\cdots+n_p\) and \(|\vec m|=m_1+\cdots+m_q\).
We use the standard convention that \(n_i=0\) implies
\(A^{(i)}\equiv0\), and \(m_j=0\) implies \(B^{(j)}\equiv0\).
Accordingly, formulas containing \((n_i-1)!\) or \((m_j-1)!\) are asserted
only when \(n_i\ge1\) or \(m_j\ge1\), respectively.

If \(|\vec n|=|\vec m|+1\), there is a vector polynomial
\begin{align*}
 (A_{\vec n,\vec m}^{(1)},\ldots,A_{\vec n,\vec m}^{(p)}),
 && \deg A_{\vec n,\vec m}^{(i)}\le n_i-1,
\end{align*}
satisfying the mixed orthogonality conditions
\begin{align}
\label{eq:OMtypeI}
\sum_{i=1}^p \int_{\Delta} A_{\vec{n},\vec{m}}^{(i)}(x)x^k\,\d\mu_{j,i}(x)=0,
&& 0\le k\le m_j-1,&& 1\le j\le q.
\end{align}
Analogously, if \(|\vec m|=|\vec n|+1\), there is a vector polynomial
\begin{align*}
 (B_{\vec n,\vec m}^{(1)},\ldots,B_{\vec n,\vec m}^{(q)}),
 && \deg B_{\vec n,\vec m}^{(j)}\le m_j-1,
\end{align*}
satisfying
\begin{align}
\label{eq:OMtypeII}
\sum_{j=1}^q \int_{\Delta} B_{\vec{n},\vec{m}}^{(j)}(x)x^k\,\d\mu_{j,i}(x)=0,
&& 0\le k\le n_i-1,&& 1\le i\le p.
\end{align}

\subsection{Mixed orthogonal polynomials for rank-one matrix measures}
Let \(\mathrm d w\) be a positive Borel measure on \(\Delta\), and consider
positive measurable weight functions
\begin{align*}
	u_1,\ldots,u_p:&\Delta\subseteq \R\rightarrow\R_+,&
	v_1,\ldots,v_q:&\Delta\subseteq \R\rightarrow\R_+.
\end{align*}
We assume that all moments used below are finite; explicitly,
\begin{align*}
 \int_\Delta |x|^r u_i(x)v_j(x)\,\mathrm d w(x)<\infty,
 && r\in\mathbb N_0,&& 1\le i\le p,&& 1\le j\le q.
\end{align*}
Define the vectors and the rank-one matrix of weights by
\(u  \coloneq(u_1,\ldots,u_p)\),
\( v\coloneq(v_1,\ldots,v_q)\),
\begin{align*}
\d \mu & \coloneq v^\top u\d w=\begin{bmatrix}
	v_1 u_1&\cdots& v_1 u_p\\
	\vdots&&\vdots\\
	v_qu_1&\cdots&v_qu_p
\end{bmatrix}\d w.
\end{align*}
Taking
\begin{align*}
A_{\vec{n},\vec{m}}(x) & \coloneq
\sum_{i=1}^p A_{\vec{n},\vec{m}}^{(i)}(x)u_i(x),
&
B_{\vec{{n}},\vec{{m}}}(x) & \coloneq \sum_{j=1}^q B_{\vec{{n}},\vec{{m}}}^{(j)}(x)v_j(x)
\end{align*}
we can rewrite \eqref{eq:OMtypeI} and \eqref{eq:OMtypeII} as:
\begin{align}
\label{MOI}
\int_\Delta x^k 
\sum_{i=1}^p A_{\vec{n},\vec{m}}^{(i)}(x)u_i(x)
v_j(x)\d w(x) & =0,&& j\in\{1,\ldots,q\},&& k\in\{0,\ldots,m_j-1\} , \\
\label{MOII}
\int_\Delta x^k
\sum_{j=1}^q B_{\vec{{n}},\vec{{m}}}^{(j)}(x)v_j(x)
u_i(x)\d w(x) & =0,&& i\in\{1,\ldots,p\},&& k\in\{0,\ldots,{n}_i-1\}.
\end{align}
For \(d\in\{p,q\}\), let
\begin{align*}
 \vec e_k^{\,(d)}
 :=(\underbrace{0,\ldots,0}_{k-1},1,
 \underbrace{0,\ldots,0}_{d-k})\in\mathbb R^d,
 && 1\le k\le d.
\end{align*}
When the dimension is determined by the surrounding multi-index, we write
simply \(\vec e_k\).

If \(u=(u_1,\ldots,u_p)\) and \(v=(v_1,\ldots,v_q)\) are AT
systems on \(\Delta\), every admissible multi-index is normal
\cite{Ulises-Sergio-Judith}. Thus the corresponding mixed forms are unique
up to a nonzero factor, and every component with positive upper degree has
maximal degree:
\begin{align*}
 \deg A_{\vec n,\vec m}^{(i)}=n_i-1,&&
 \deg B_{\vec n,\vec m}^{(j)}=m_j-1
\end{align*}
whenever \(n_i>0\) and \(m_j>0\), respectively.

The next proposition records the biorthogonality relation used below to
isolate the recurrence coefficients.
\begin{pro}[Biorthogonality]
\label{prop:biorthogonality}
Assume that \(u=(u_1,\ldots,u_p)\) and \(v=(v_1,\ldots,v_q)\) are AT
systems on \(\Delta\). Let \(A_{\vec n,\vec m}\) be a type-I form, so
\(|\vec n|=|\vec m|+1\), and let
\(B_{\vec{\mathfrak n},\vec{\mathfrak m}}\) be a type-II form, so
\(|\vec{\mathfrak m}|=|\vec{\mathfrak n}|+1\). Then
\begin{align}
	\label{MixedBiorthogonality}
	\int_\Delta
	A_{\vec n,\vec m}(x)
	B_{\vec{\mathfrak n},\vec{\mathfrak m}}(x)\,\d w(x)
	=0
	&&\text{whenever}&&
	\vec n\leq\vec{\mathfrak n}
	&&\text{or}&&
	\vec{\mathfrak m}\leq\vec m.
\end{align}
Moreover, 
this integral
 is nonzero in either adjacent case
\begin{align}
\label{eq:bio-adjacent}
\vec{\mathfrak m}=\vec m+\vec e_{j_0}
&&\text{or}&&
\vec n=\vec{\mathfrak n}+\vec e_{i_0}.
\end{align}
\end{pro}

\begin{proof}
Write
\begin{align*}
 A_{\vec n,\vec m}^{(i)}(x)
 =\sum_{\ell=0}^{n_i-1}A_{\vec n,\vec m}^{(i)}[\ell]x^\ell .
\end{align*}
If \(\vec n\leq\vec{\mathfrak n}\), then
\begin{align*}
 \int_\Delta A_{\vec n,\vec m}(x)
 B_{\vec{\mathfrak n},\vec{\mathfrak m}}(x)\,\d w(x)
 =\sum_{i=1}^p\sum_{\ell=0}^{n_i-1}
 A_{\vec n,\vec m}^{(i)}[\ell]
 \int_\Delta x^\ell u_i(x)
 B_{\vec{\mathfrak n},\vec{\mathfrak m}}(x)\,\d w(x)=0
\end{align*}
by the type-II conditions. Likewise, expanding
\begin{align*}
 B_{\vec{\mathfrak n},\vec{\mathfrak m}}^{(j)}(x)
 =\sum_{\ell=0}^{\mathfrak m_j-1}
 B_{\vec{\mathfrak n},\vec{\mathfrak m}}^{(j)}[\ell]x^\ell
\end{align*}
and using the type-I conditions gives
\begin{align*}
 \int_\Delta A_{\vec n,\vec m}(x)
 B_{\vec{\mathfrak n},\vec{\mathfrak m}}(x)\,\d w(x)=0
 &&\text{if}&&
 \vec{\mathfrak m}\leq\vec m.
\end{align*}

Suppose now that
\(\vec{\mathfrak m}=\vec m+\vec e_{j_0}\). The type-I conditions annihilate
every term except the leading term of the \(j_0\)-component:
\begin{align*}
 \int_\Delta A_{\vec n,\vec m}(x)
 B_{\vec{\mathfrak n},\vec{\mathfrak m}}(x)\,\d w(x)
 =B_{\vec{\mathfrak n},\vec{\mathfrak m}}^{(j_0)}[m_{j_0}]
 \int_\Delta x^{m_{j_0}}v_{j_0}(x)
 A_{\vec n,\vec m}(x)\,\d w(x).
\end{align*}
The coefficient
\(B_{\vec{\mathfrak n},\vec{\mathfrak m}}^{(j_0)}[m_{j_0}]\)
is nonzero by normality. Choose \(i_0\) with \(n_{i_0}>0\). If the remaining
moment vanished, then \(A_{\vec n,\vec m}\), regarded as an element of the
larger polynomial space corresponding to
\( (\vec n+\vec e_{i_0},\vec m+\vec e_{j_0}) \),
would satisfy all the type-I conditions for this enlarged index. This
contradicts normality: its \(i_0\)-component has degree at most
\(n_{i_0}-1\), whereas every nonzero form for the enlarged index has
\(i_0\)-component of degree~\(n_{i_0}\).

Finally, suppose that
\(\vec n=\vec{\mathfrak n}+\vec e_{i_0}\). The type-II conditions now give
\begin{align*}
 \int_\Delta A_{\vec n,\vec m}(x)
 B_{\vec{\mathfrak n},\vec{\mathfrak m}}(x)\,\d w(x)
 =A_{\vec n,\vec m}^{(i_0)}[\mathfrak n_{i_0}]
 \int_\Delta x^{\mathfrak n_{i_0}}u_{i_0}(x)
 B_{\vec{\mathfrak n},\vec{\mathfrak m}}(x)\,\d w(x).
\end{align*}
The displayed coefficient is nonzero by normality. If the remaining moment
vanished, then, for any \(j_0\), the original form
\(B_{\vec{\mathfrak n},\vec{\mathfrak m}}\) would satisfy all the type-II
conditions for
\( (\vec{\mathfrak n}+\vec e_{i_0},
 \vec{\mathfrak m}+\vec e_{j_0}) \),
while its \(j_0\)-component would still have degree at most
\(\mathfrak m_{j_0}-1\). Normality of the enlarged index requires that
degree to be \(\mathfrak m_{j_0}\), which is again a contradiction.
\end{proof}

\section{Recurrence relations}

Consider the vectors
\(\vec{1}_p\coloneq(
{1,\ldots,1}
)\in\R^p\),
\( \vec{0}_p\coloneq(
{0,\ldots,0}
)\in\R^p\).
Let \(\pi_s\) and \(\pi_{\mathfrak s}\) be permutations of
\(\{1,\ldots,p\}\), and let \(\pi_t\) and \(\pi_{\mathfrak t}\) be
permutations of \(\{1,\ldots,q\}\). Define
\begin{align}
	&\vec{s}_{rp+i}\coloneq r\vec{1}_p+\sum_{k=1}^i\vec{e}_{\pi_s(k)}, &&
	\vec{\mathfrak{s}}_{rp+i}\coloneq r\vec{1}_p+\sum_{k=1}^i\vec{e}_{\pi_\mathfrak{s}(k)}, && r\in\N_0, && i\in\{0,\ldots,p-1\},\\
	&\vec{t}_{rq+j}\coloneq r\vec{1}_q+\sum_{k=1}^j\vec{e}_{\pi_t(k)}, &&
	\vec{\mathfrak{t}}_{rq+j}\coloneq r\vec{1}_q+\sum_{k=1}^j\vec{e}_{\pi_\mathfrak{t}(k)}, && r\in\N_0, && j\in\{0,\ldots,q-1\}.
\end{align}

The vectors \(\vec{s}_{i},\vec{\mathfrak{s}}_i\) belong to \(\N_0^p\),
whereas \(\vec{t}_{j},\vec{\mathfrak{t}}_j\) belong to \(\N_0^q\). They define chains starting at
\(\vec{s}_0=\vec{\mathfrak{s}}_0=\vec{0}_p\),
\(\vec{t}_0=\vec{\mathfrak{t}}_0=\vec{0}_q\).
At each step the modulus increases by one, with the component order determined
by \(\pi_s\),~\(\pi_{\mathfrak s}\),~\(\pi_t\), and
\(\pi_{\mathfrak t}\).

The multi-index \(\vec n\pm\vec s_{rp+i}\) is obtained by first adding or
subtracting \(r\vec 1_p\) and then changing by \(1\) the~\(i\)~entries
\(	\{{n}_{\pi_s(1)},{n}_{\pi_s(2)},\ldots,{n}_{\pi_s(i)}\} \).
Similarly, for \( \vec{n}\pm\vec{\mathfrak{s}}_{rp+i}\) with \(\pi_{\mathfrak{s}}\).
Consequently,
\(	|\vec{n}\pm\vec{s}_{i}|=|\vec{n}\pm\vec{\mathfrak{s}}_{i}|=|\vec{n}|\pm i\).
In the same way,
 \(\vec m\pm\vec t_{rq+j}\) is obtained by changing the \(j\)
entries
\(	\{{m}_{\pi_t(1)},{m}_{\pi_t(2)},\ldots,{m}_{\pi_t(j)}\}\).
Furthermore, for \( \vec{m}\pm\vec{\mathfrak{t}}_{rq+j}\) with \(\pi_{\mathfrak{t}}\). Consequently,
\(	|\vec{m}\pm\vec{t}_{j}|=|\vec{m}\pm\vec{\mathfrak{t}}_{j}|=|\vec{m}|\pm j \).

The four admissible index paths introduced above lead to the following mixed
recurrence relations.
\begin{pro}
Assume that \(u=(u_1,\ldots,u_p)\) and \(v=(v_1,\ldots,v_q)\) are AT
systems on \(\Delta\).
Assume in this proposition that all components of \(\vec n\) and \(\vec m\)
are positive and that every shifted multi-index displayed below is
nonnegative.
Fix an arbitrary nonzero
normalization for every form on the four chains,
then the linear forms \eqref{MOII} and \eqref{MOI} satisfy respective recurrence relations
\begin{align}
	\label{MixedRII}
	x B_{\vec{n},\vec{m}}(x)=\sum_{j=1}^q b^j_{\vec{n},\vec{m}}\, B_{\vec{n}+\vec{\mathfrak{s}}_j,\vec{m}+\vec{t}_j}(x)+b^0_{\vec{n},\vec{m}}\,B_{\vec{n},\vec{m}}(x)+\sum_{i=1}^p b^{-i}_{\vec{n},\vec{m}}\, B_{\vec{n}-\vec{s}_i,\vec{m}-\vec{\mathfrak{t}}_i}(x) , \\
	\label{MixedRI}
	x A_{\vec{n},\vec{m}}(x)=\sum_{j=1}^q a^{-j}_{\vec{n},\vec{m}}\, A_{\vec{n}-\vec{\mathfrak{s}}_j,\vec{m}-\vec{t}_j}(x)+a^0_{\vec{n},\vec{m}}\,A_{\vec{n},\vec{m}}(x)+\sum_{i=1}^p a^{i}_{\vec{n},\vec{m}}\, A_{\vec{n}+\vec{s}_i,\vec{m}+\vec{\mathfrak{t}}_i}(x) ,
\end{align}
where
\begin{align}
	\label{eq:anversusbn}
\left\{\begin{aligned}
a^{-j}_{\vec{n},\vec{m}}&=b^{j}_{\vec{n}-\vec{\mathfrak{s}}_{j+1},\vec{m}-\vec{t}_{j-1}}\,
{
	\dfrac{\int_\Delta B_{\vec{n}-\vec{\mathfrak{s}}_{j+1}+\vec{\mathfrak{s}}_j,\vec{m}-\vec{t}_{j-1}+\vec{t}_j}(x) A_{\vec{n},\vec{m}}(x) \d w(x)}{\int_\Delta B_{\vec{n}-\vec{\mathfrak{s}}_{j+1},\vec{m}-\vec{t}_{j-1}}(x) A_{\vec{n}-\vec{\mathfrak{s}}_j,\vec{m}-\vec{t}_j}(x) \d w(x)}},\quad j\in\{1,\ldots,q\},\\[4pt]
a^0_{\vec{n},\vec{m}}&=b^0_{\vec{n}-\vec{\mathfrak{s}}_1,\vec{m}+\vec{\mathfrak{t}}_1},\\[4pt]
a^{i}_{\vec{n},\vec{m}}&=b^{-i}_{\vec{n}+\vec{s}_{i-1},\vec{m}+\vec{\mathfrak{t}}_{i+1}}\,
{
	\dfrac{\int_\Delta B_{\vec{n}-\vec{s}_{i}+\vec{s}_{i-1},\vec{m}-\vec{\mathfrak{t}}_{i}+\vec{\mathfrak{t}}_{i+1}}(x) A_{\vec{n},\vec{m}}(x) \d w(x)}{\int_\Delta B_{\vec{n}+\vec{s}_{i-1},\vec{m}+\vec{\mathfrak{t}}_{i+1}}(x) A_{\vec{n}+\vec{s}_i,\vec{m}+\vec{\mathfrak{t}}_i}(x) \d w(x)}},\quad i\in\{1,\ldots,p\},
\end{aligned}\right.
\end{align}
and
\begin{align}
	\label{MixedRecurrence}
\left\{\begin{aligned}
b^{j}_{\vec{n},\vec{m}}&=
\dfrac{\int_\Delta x\, A_{\vec{n}+\vec{\mathfrak{s}}_{j+1},\vec{m}+\vec{t}_{j-1}}(x) B_{\vec{n},\vec{m}}(x) \d w(x)}{
	\int_\Delta A_{\vec{n}+\vec{\mathfrak{s}}_{j+1},\vec{m}+\vec{t}_{j-1}}(x) B_{\vec{n}+\vec{\mathfrak{s}}_j,\vec{m}+\vec{t}_j}(x) \d w (x)},\quad j\in\{1,\ldots,q\},\\[4pt]
b^{0}_{\vec{n},\vec{m}}&=
\dfrac{\int_\Delta x\, A_{\vec{n}+\vec{\mathfrak{s}}_{1},\vec{m}-\vec{\mathfrak{t}}_{1}}(x) B_{\vec{n},\vec{m}}(x) \d w(x)}{\int_\Delta A_{\vec{n}+\vec{\mathfrak{s}}_{1},\vec{m}-\vec{\mathfrak{t}}_{1}}(x) B_{\vec{n},\vec{m}}(x) \d w(x)},\\[4pt]
b^{-i}_{\vec{n},\vec{m}}&=
\dfrac{\int_\Delta x\, A_{\vec{n}-\vec{s}_{i-1},\vec{m}-\vec{\mathfrak{t}}_{i+1}}(x) B_{\vec{n},\vec{m}}(x) \d w(x)}{\int_\Delta A_{\vec{n}-\vec{s}_{i-1},\vec{m}-\vec{\mathfrak{t}}_{i+1}}(x) B_{\vec{n}-\vec{s}_i,\vec{m}-\vec{\mathfrak{t}}_i}(x) \d w(x)},\quad i\in\{1,\ldots,p\}.
\end{aligned}\right.
\end{align}
\end{pro}

\begin{proof}
\emph{Existence and uniqueness.}
Let \(\mathcal V\) be the space of forms
\(F=\sum_{j=1}^qP_jv_j\) with \(\deg P_j\le m_j\) that satisfy
\begin{align*}
 \int_\Delta x^kF(x)u_i(x)\,\d w(x)=0,
 \qquad 0\le k\le n_i-2,\quad 1\le i\le p,
\end{align*}
with an empty range when \(n_i=1\). Write
\(P_j(x)=\sum_{\ell=0}^{m_j}c_{j,\ell}x^\ell\). The coefficient matrix of
the displayed conditions is
\begin{align*}
 \left[
 \int_\Delta x^{k+\ell}u_i(x)v_j(x)\,\d w(x)
 \right]_{\substack{1\leq i\leq p,\ 0\leq k\leq n_i-2\\
 1\leq j\leq q,\ 0\leq\ell\leq m_j}} .
\end{align*}
The perfectness theorem for the two AT systems gives full row rank for this
matrix \cite{Ulises-Sergio-Judith}. Hence the
\(|\vec n|-p\) displayed conditions are independent. The coefficient space
before imposing them has dimension
\begin{align*}
 \sum_{j=1}^q(m_j+1)=|\vec m|+q=|\vec n|+q+1.
\end{align*}
Consequently,
\begin{align*}
 \dim\mathcal V=(|\vec n|+q+1)-(|\vec n|-p)=p+q+1.
\end{align*}

Consider the following \(p+q+1\) forms:
\begin{align*}
\mathcal B_{\vec n,\vec m}:={}&
\big\{B_{\vec n+\vec{\mathfrak s}_j,\vec m+\vec t_j}:
1\le j\le q\big\} \cup \{B_{\vec n,\vec m}\}
\cup\big\{B_{\vec n-\vec s_i,\vec m-\vec{\mathfrak t}_i}:
1\le i\le p\big\}.
\end{align*}
Every form in this collection belongs to \(\mathcal V\). For a forward
form, the degree increase in \(\vec m+\vec t_j\) is at most one in each
component, while its type-II orthogonality conditions contain those
defining \(\mathcal V\). The central form satisfies stronger degree bounds,
and the backward forms satisfy the required moment conditions because the
corresponding components of \(\vec n-\vec s_i\) have decreased by at most
one.

The collection is linearly independent. Indeed, pair a linear combination
of its elements successively~with
\( A_{\vec n+\vec{\mathfrak s}_{j+1},\vec m+\vec t_{j-1}}\),
\( 1\le j\le q\),
then with \(A_{\vec n+\vec{\mathfrak s}_1,\vec m-\vec{\mathfrak t}_1}\),
and finally with
\( A_{\vec n-\vec s_{i-1},\vec m-\vec{\mathfrak t}_{i+1}}\),
\( 1\le i\le p\).
The biorthogonality relations annihilate all terms except the corresponding
forward, central, or backward pivot. The precise vanishing relations are
\eqref{eq:vanishuno}--\eqref{eq:vanishcuatro} below, and every pivot is
nonzero by Proposition~\ref{prop:biorthogonality}. Thus
\(\mathcal B_{\vec n,\vec m}\) is a basis of \(\mathcal V\).

Finally, \(xB_{\vec n,\vec m}\in\mathcal V\): its component degrees are at
most \(m_j\), and multiplication by \(x\) shifts every type-II moment
condition by one. Therefore it has a unique expansion in the preceding
basis,
\begin{align}
\label{eq:recurrencetypeII}
x B_{\vec{n},\vec{m}}(x) = \sum_{j=1}^q b^j_{\vec{n},\vec{m}}\,
B_{\vec{n}+\vec{\mathfrak{s}}_j,\vec{m}+\vec{t}_j}(x)
b^0_{\vec{n},\vec{m}}B_{\vec{n},\vec{m}}(x)
+\sum_{i=1}^p b^{-i}_{\vec{n},\vec{m}}\,
B_{\vec{n}-\vec{s}_i,\vec{m}-\vec{\mathfrak{t}}_i}(x).
\end{align}
This proves existence and uniqueness of the type-II recurrence. The type-I
recurrence~\eqref{MixedRI} follows by the same argument with the two families interchanged.

\medskip
\noindent\emph{Forward coefficients.}
Fix \(j_0\in\{1,\ldots,q\}\). We have
\begin{align*}
\int_\Delta A_{\vec{n}+\vec{\mathfrak{s}}_{j_0+1},\vec{m}+\vec{t}_{j_0-1}}(x) B_{\vec{n}+\vec{\mathfrak{s}}_j,\vec{m}+\vec{t}_j}(x)\d w(x) = 0, && j \neq j_0 ,
\end{align*}
taking into account
\begin{align}
\label{eq:vanishuno}
\vec{n}+\vec{\mathfrak{s}}_{j_0+1} \leq \vec{n}+\vec{\mathfrak{s}}_{j}, && j \in \{j_0 +1, \ldots, q\}, &&
 \vec{m}+\vec{t}_{j}\leq \vec{m}+\vec{t}_{j_0-1}, && j \in \{1, \ldots, j_0-1\} ,
\end{align}
and the biorthogonal relations \eqref{MixedBiorthogonality}.
Noticing that \( \vec{m}-\vec{\mathfrak{t}}_i \leq \vec{m}+\vec{t}_{j_0-1} \) and again using the biorthogonal relations~\eqref{MixedBiorthogonality}, it holds
\begin{align}
\label{eq:vanishdois}
\int_\Delta A_{\vec{n}+\vec{\mathfrak{s}}_{j_0+1},\vec{m}+\vec{t}_{j_0-1}}(x) B_{\vec{n}-\vec{s}_i,\vec{m}-\vec{\mathfrak{t}}_i}(x)\d w(x) = 0, && i >0 .
\end{align}
Moreover,
using \eqref{MixedBiorthogonality}
\begin{align*}
\int_\Delta A_{\vec{n}+\vec{\mathfrak{s}}_{j_0+1},\vec{m}+\vec{t}_{j_0-1}}(x) B_{\vec{n}+\vec{\mathfrak{s}}_{j_0},\vec{m}+\vec{t}_{j_0}}(x)\d w(x) \ne 0 .
\end{align*}
Multiplying \eqref{eq:recurrencetypeII} by
\(A_{\vec n+\vec{\mathfrak s}_{j_0+1},\vec m+\vec t_{j_0-1}}\),
integrating, and using \eqref{eq:vanishuno} and \eqref{eq:vanishdois}
gives the formula for \(b^{j_0}_{\vec n,\vec m}\).

\medskip
\noindent\emph{Diagonal coefficient.}
For the diagonal coefficient, pair \eqref{eq:recurrencetypeII} with
\(A_{\vec n+\vec{\mathfrak s}_1,\vec m-\vec{\mathfrak t}_1}\).
Part (a) of Proposition~\ref{prop:biorthogonality} annihilates every forward
term, and part (b) annihilates every backward term. The remaining pairing
with \(B_{\vec n,\vec m}\) is nonzero by
\eqref{eq:bio-adjacent}. Division by this pairing gives the second line of
\eqref{MixedRecurrence}.

\medskip
\noindent\emph{Backward coefficients.}
Fix \(i_0\in\{1,\ldots,p\}\). We have
\begin{align*}
\int_\Delta A_{\vec{n}-\vec{s}_{i_0-1},\vec{m}-\vec{\mathfrak{t}}_{i_0+1}}(x) B_{\vec{n}-\vec{s}_i,\vec{m}-\vec{\mathfrak{t}}_i}(x) \d w(x) = 0, && i \ne i_0,
\end{align*}
taking also into account
\begin{align}
\label{eq:vanishtres}
\vec{n}-\vec{s}_{i_0-1} \leq \vec{n}-\vec{s}_i, && i \in \{1, \ldots, i_0-1\}, &&
 \vec{m}-\vec{\mathfrak{t}}_i \leq \vec{m}-\vec{\mathfrak{t}}_{i_0+1}, && i \geq i_0+1.
\end{align}
and the biorthogonal relations \eqref{MixedBiorthogonality}.
Noticing that \( \vec{n}-\vec{s}_{i_0-1} \leq \vec{n}+\vec{\mathfrak{s}}_j \) and again using the biorthogonal relations~\eqref{MixedBiorthogonality},
it holds
\begin{align}
\label{eq:vanishcuatro}
\int_\Delta A_{\vec{n}-\vec{s}_{i_0-1},\vec{m}-\vec{\mathfrak{t}}_{i_0+1}}(x) B_{\vec{n}+\vec{\mathfrak{s}}_j,\vec{m}+\vec{t}_j}(x)\d w(x) = 0, && j \in \{1,\ldots,q\}
\end{align}
Moreover,
\begin{align*}
\int_\Delta A_{\vec{n}-\vec{s}_{i_0-1},\vec{m}-\vec{\mathfrak{t}}_{i_0+1}}(x) B_{\vec{n}-\vec{s}_{i_0},\vec{m}-\vec{\mathfrak{t}}_{i_0}}(x) \d w(x) \ne 0 
 .
\end{align*}
Multiplying \eqref{eq:recurrencetypeII} by
\(A_{\vec n-\vec s_{i_0-1},\vec m-\vec{\mathfrak t}_{i_0+1}}\),
integrating, and using \eqref{eq:vanishtres} and
\eqref{eq:vanishcuatro} gives the formula for
\(b^{-i_0}_{\vec n,\vec m}\) in \eqref{MixedRecurrence}.

The type-I coefficients follow by duality. Substituting the shifted
type-II coefficient formulas gives \eqref{eq:anversusbn}.
\end{proof}

Write the polynomial components of the mixed forms as
\begin{align*}
A^{(i)}_{\vec{n},\vec{m}}(x)=\sum_{l=0}^{n_i-1}	A^{(i)}_{\vec{n},\vec{m}}[l]\,x^l,\quad i\in\{1,\ldots,p\}, &&
B^{(j)}_{\vec{\mathfrak{n}},\vec{\mathfrak{m}}}(x)=\sum_{l=0}^{\mathfrak{m}_j-1}	B^{(j)}_{\vec{\mathfrak{n}},\vec{\mathfrak{m}}}[l]\,x^l,\quad j\in\{1,\ldots,q\}.
\end{align*}
For later boundary formulas, set \(A^{(i)}_{\vec n,\vec m}[l]=0\) and
\(B^{(j)}_{\vec n,\vec m}[l]=0\) whenever \(l\) lies outside the indicated
coefficient~range.

Using these coefficients and the orthogonality conditions
\eqref{MOII}--\eqref{MOI}, we can rewrite the recurrence coefficients in
\eqref{MixedRecurrence} as follows:
\begin{align}\label{ReMixedRecurrencej}
&	\begin{aligned}
	b^{j}_{\vec{n},\vec{m}}
	&=\dfrac{\int_\Delta x\, A_{\vec{n}+\vec{\mathfrak{s}}_{j+1},\vec{m}+\vec{t}_{j-1}}(x) B_{\vec{n},\vec{m}}(x) \d w(x)}{\int_\Delta A_{\vec{n}+\vec{\mathfrak{s}}_{j+1},\vec{m}+\vec{t}_{j-1}}(x) B_{\vec{n}+\vec{\mathfrak{s}}_j,\vec{m}+\vec{t}_j}(x) \d w(x)}\\
&	=\sum_{k=j}^q\dfrac{B^{(\pi_t(k))}_{\vec{n},\vec{m}}[m_{\pi_t(k)}-1]\int_\Delta x^{m_{\pi_t(k)}}A_{\vec{n}+\vec{\mathfrak{s}}_{j+1},\vec{m}+\vec{t}_{j-1}}(x)v_{\pi_t(k)}(x)\d w(x)}
	{B^{(\pi_t(j))}_{\vec{n}+\vec{\mathfrak{s}}_j,\vec{m}+\vec{t}_j}[m_{\pi_t(j)}]\int_\Delta x^{m_{\pi_t(j)}}A_{\vec{n}+\vec{\mathfrak{s}}_{j+1},\vec{m}+\vec{t}_{j-1}}(x)v_{\pi_t(j)}(x)\d w(x)},
	\end{aligned}\\[4pt]	\label{ReMixedRecurrence0}
&	\begin{aligned}
	b^{0}_{\vec{n},\vec{m}}&=
	\dfrac{\int_\Delta x\, A_{\vec{n}+\vec{\mathfrak{s}}_{1},\vec{m}-\vec{\mathfrak{t}}_{1}}(x) B_{\vec{n},\vec{m}}(x) \d w(x)}{\int_\Delta A_{\vec{n}+\vec{\mathfrak{s}}_{1},\vec{m}-\vec{\mathfrak{t}}_{1}}(x) B_{\vec{n},\vec{m}}(x) \d w(x)}\\
	&=\sum_{k=1}^q \dfrac{B^{(k)}_{\vec{n},\vec{m}}[m_k-1]}
{B^{(\pi_{\mathfrak{t}}(1))}_{\vec{n},\vec{m}}[m_{\pi_{\mathfrak{t}}(1)}-1]}\,\dfrac{\int_{\Delta}x^{m_k}A_{\vec{n}+\vec{\mathfrak{s}}_1,\vec{m}-\vec{\mathfrak{t}}_1}(x)v_k(x)\d w(x)}{\int_\Delta x^{m_{\pi_\mathfrak{t}(1)}-1}A_{\vec{n}+\vec{\mathfrak{s}}_1,\vec{m}-\vec{\mathfrak{t}}_1}(x)v_{\pi_{\mathfrak{t}}(1)}(x)\d w(x)	}
	+\dfrac{B^{(\pi_{\mathfrak{t}}(1))}_{\vec{n},\vec{m}}[m_{\pi_{\mathfrak{t}}(1)}-2]}{B^{(\pi_{\mathfrak{t}}(1))}_{\vec{n},\vec{m}}[m_{\pi_{\mathfrak{t}}(1)}-1]}\\
	&=\sum_{k=1}^p \dfrac{A^{(k)}_{\vec{n}+\vec{\mathfrak{s}}_{1},\vec{m}-\vec{\mathfrak{t}}_{1}}[n_k-1]}
	{A^{(\pi_{\mathfrak{s}}(1))}_{\vec{n}+\vec{\mathfrak{s}}_{1},\vec{m}-\vec{\mathfrak{t}}_{1}}[n_{\pi_{\mathfrak{s}}(1)}]}\,\dfrac{\int_{\Delta}x^{n_k}B_{\vec{n},\vec{m}}(x)u_k(x)\d w(x)}{\int_\Delta x^{n_{\pi_\mathfrak{s}(1)}}B_{\vec{n},\vec{m}}(x)u_{\pi_{\mathfrak{s}}(1)}(x)\d w(x)	}
	+\dfrac{\int_\Delta x^{n_{\pi_\mathfrak{s}(1)}+1}B_{\vec{n},\vec{m}}(x)u_{\pi_{\mathfrak{s}}(1)}(x)\d w(x)}{\int_\Delta x^{n_{\pi_\mathfrak{s}(1)}}B_{\vec{n},\vec{m}}(x)u_{\pi_{\mathfrak{s}}(1)}(x)\d w(x)	},
	\end{aligned}\\[4pt]	\label{ReMixedRecurrencei}
&	\begin{aligned}
	b^{-i}_{\vec{n},\vec{m}}&=
	\dfrac{\int_\Delta x\, A_{\vec{n}-\vec{s}_{i-1},\vec{m}-\vec{\mathfrak{t}}_{i+1}}(x) B_{\vec{n},\vec{m}}(x) \d w(x)}{\int_\Delta A_{\vec{n}-\vec{s}_{i-1},\vec{m}-\vec{\mathfrak{t}}_{i+1}}(x) B_{\vec{n}-\vec{s}_i,\vec{m}-\vec{\mathfrak{t}}_i}(x) \d w(x)}\\
&	=\sum_{k=i}^p\dfrac{	A^{(\pi_s(k))}_{\vec{n}-\vec{s}_{i-1},\vec{m}-\vec{\mathfrak{t}}_{i+1}}[n_{\pi_s(k)}-1]\int_\Delta x^{n_{\pi_s(k)}}B_{\vec{n},\vec{m}}(x)u_{\pi_s(k)}(x)\d w(x)}{
	A^{(\pi_s(i))}_{\vec{n}-\vec{s}_{i-1},\vec{m}-\vec{\mathfrak{t}}_{i+1}}[n_{\pi_s(i)}-1]\int_\Delta x^{n_{\pi_s(i)}-1}B_{\vec{n}-\vec{s}_{i},\vec{m}-\vec{\mathfrak{t}}_{i}}(x)u_{\pi_s(i)}(x)\d w(x)}.
\end{aligned}
\end{align}

\section{The step line}

We say that a pair of multi-indices \((\vec n,\vec m)\) lies on the step line
if it has the form
\begin{align}
	\label{SL}
	\vec{n}&=(\underbrace{\mathfrak r+1,\ldots,\mathfrak r+1}_{\mathfrak l\text{ times}},\underbrace{\mathfrak r,\ldots,\mathfrak r}_{p-\mathfrak l\text{ times}}), && \mathfrak r\in\N_0, && \mathfrak l\in\{0,\ldots,p-1\},\\
	\vec{m}&=(\underbrace{r+1,\ldots,r+1}_{l\text{ times}},\underbrace{r,\ldots,r}_{q-l\text{ times}}), && r\in\N_0, && l\in\{0,\ldots,q-1\}
	.
\end{align}
Thus \(|\vec{n}|=p\mathfrak r+\mathfrak l\) and
\(|\vec{m}|=qr+l\). Equivalently, their components are
\begin{align*}
n_i  = \left\lceil \frac{|\vec{n}|+1-i}{p}\right\rceil
&& \text{and} &&
m_j = \left\lceil \frac{|\vec{m}|+1-j}{q}\right\rceil .
\end{align*}
We use the following compatible normalizations.

If \(|\vec{m}|=|\vec{n}|+1\), write
\(|\vec n|=qr+l\), where \(l\in\{0,\ldots,q-1\}\), and normalize the type-II
polynomials in~\eqref{MOII}~by
\begin{align}
	\label{NormIISL}
	B^{(l+1)}_{\vec{n},\vec{m}}[r]=1.
\end{align}
Thus \(B^{(l+1)}_{\vec{n},\vec{m}}\) is monic of degree
\(r\).

If, on the other hand, \(|\vec{n}|=|\vec{m}|+1\), write
\(|\vec m|=qr+l\), where \(l\in\{0,\ldots,q-1\}\), and use the type-I
normalization
\begin{align}
	\label{NormISL}
		\int_{\Delta} x^r A_{\vec{n},\vec{m}}(x)v_{l+1}(x)\d w(x)=1.
\end{align}

The step-line multi-indices in \eqref{SL} can be ordered so that each modulus
\(|\vec n|=p\mathfrak r+\mathfrak l\) and
\(|\vec m|=qr+l\) increases by one along its corresponding sequence.
We may therefore relabel the polynomials by the modulus of the corresponding multi-index.
Thus, we can relabel the polynomials \eqref{MOII} as
\begin{align*}
	B_n^{(j)}= B_{p\mathfrak r + \mathfrak l}^{(j)}&\coloneq B^{(j)}_{\vec{n},\vec{m}},&& j\in\{1,\ldots,q\},&& n\in\N_0,
\end{align*}
where \(\vec n\in\N_0^p\) and \(\vec m\in\N_0^q\) are the corresponding
step-line multi-indices satisfying
\(|\vec n|=n\),
\( |\vec m|=n+1\).
Similarly, we can relabel the polynomials \eqref{MOI} as
\begin{align*}
	A_m^{(i)}= A_{qr + l}^{(i)}&\coloneq A^{(i)}_{\vec{n},\vec{m}}, &&  i\in\{1,\ldots,p\}, && m\in\N_0,
\end{align*}
where \(\vec n\in\N_0^p\) and \(\vec m\in\N_0^q\) are the corresponding
step-line multi-indices satisfying
\(|\vec n|=m+1\),
\(|\vec m|=m\).
In~the same way, relabel the corresponding linear forms by
\begin{align}
	\label{LFSL}
	B_n(x)&=\sum_{j=1}^q B_n^{(j)}(x)v_j(x),
	&
	A_m(x)&=\sum_{i=1}^p A_m^{(i)}(x)u_i(x),
	& n,m\in\N_0.
\end{align}
Equation~\eqref{MixedBiorthogonality} gives the biorthogonality condition
\begin{align}
	\label{BOSL}
	\int_{\Delta} B_{n}(x)A_{m}(x)\d w(x)=\delta_{n,m}.
\end{align}
For the step-line multi-indices in \eqref{SL}, write
\begin{align*}
 n+1=qr_+ +l_+,
 && r_+\in\mathbb N_0,
 && l_+\in\{0,\ldots,q-1\}.
\end{align*}
The remainder \(l_+\), rather than \(l\), determines the two paths in the
\(q\)-component type-II multi-index, because that multi-index has modulus
\(n+1\). Consider the following permutations:
\begin{align}
	\pi_s=&\begin{bNiceMatrix}
		1 & 2 & \cdots & \mathfrak l & \mathfrak l+1 & \mathfrak l+2 & \cdots & p \\
		\mathfrak l & \mathfrak l-1 & \cdots & 1 & p & p-1 &\cdots & \mathfrak l+1
	\end{bNiceMatrix},\\ \pi_{\mathfrak s}=&\begin{bNiceMatrix}
	1 & 2 & \cdots & p-\mathfrak l & p-\mathfrak l+1 & p-\mathfrak l+2 & \cdots & p \\
	\mathfrak l+1 & \mathfrak l+2 & \cdots & p & 1 & 2 &\cdots & \mathfrak l
	\end{bNiceMatrix},\\
	\pi_t=&\begin{bNiceMatrix}
		1 & 2 & \cdots & q-l_+ & q-l_++1 & q-l_++2 & \cdots & q \\
		l_++1 & l_++2 & \cdots & q & 1 & 2 &\cdots & l_+
	\end{bNiceMatrix},\\ \pi_{\mathfrak t}=&\begin{bNiceMatrix}
	1 & 2 & \cdots & l_+ & l_++1 & l_++2 & \cdots & q \\
	l_+ & l_+-1 & \cdots & 1 & q & q-1 &\cdots & l_++1
	\end{bNiceMatrix}.
\end{align}
With these choices,
\(
(\vec n+\vec{\mathfrak s}_j,\vec m+\vec t_j)
\)
is the type-II step-line pair associated with \(B_{n+j}\), and
\(
(\vec n-\vec s_i,\vec m-\vec{\mathfrak t}_i)
\)
is the pair associated with \(B_{n-i}\) whenever it is componentwise
nonnegative. Empty blocks in the displayed permutations are omitted.

Using these quantities, we relabel the recurrence coefficients in
\eqref{MixedRII} as
\begin{align*}
	b^k_n=b^k_{p\mathfrak r+\mathfrak l}\coloneq b^k_{\vec{n},\vec{m}}, &&  k\in\{-p,\ldots,-1,0,1,\ldots,q\}, && n\in\N_0,
\end{align*}
where \(\vec n\in\N_0^p\) and \(\vec m\in\N_0^q\) are the corresponding
step-line multi-indices satisfying
\(|\vec n|=n\),
\( |\vec m|=n+1 \).
The~type-II recurrence \eqref{MixedRII} becomes
\begin{align}
	\label{MixedRIISL}
		x B_{n}(x)&=\sum_{j=1}^q b^j_{n}\, B_{n+j}(x)+b^0_{n}\,B_{n}(x)+\sum_{i=1}^p b^{-i}_{n}\, B_{n-i}(x), & &n\in\N_0.
		\end{align}
The analogous type-I recurrence is
\begin{align*}
		x A_{m}(x)&=\sum_{j=1}^q a_m^{-j}\, A_{m-j}(x)+a^0_{m}\,A_{m}(x)+\sum_{i=1}^p a^{i}_{m}\, A_{m+i}(x), & &m\in\N_0.
\end{align*}
Biorthogonality \eqref{BOSL} gives
\begin{align*}
a_m^{-j} = b^{j}_{m-j}, && a^0_{m} = b^0_{m}, && \text{and} && a^{i}_{m} = b^{-i}_{m+i}
\end{align*}
and therefore
\begin{align}
\label{MixedRISL}
x A_{m}(x)&=\sum_{j=1}^q b^{j}_{m-j}\, A_{m-j}(x)+b^0_{m}\,A_{m}(x)+\sum_{i=1}^p b^{-i}_{m+i}\, A_{m+i}(x),\quad &m\in\N_0.
\end{align}
Here \(A_n=B_n=0\) for every \(n<0\).

For \(0\le n<p\), set \(b_n^{-i}=0\) when \(i>n\). The proof of the preceding
proposition applies after the negative-index forms are omitted. Indeed, the
space determined by the same degree bounds and shifted moment conditions has
dimension \(q+1+n\), while
\( \{B_{n+1},\ldots,B_{n+q},B_n,B_{n-1},\ldots,B_0\} \)
consists of \(q+1+n\) biorthogonally independent elements. Since \(xB_n\)
belongs to this space, its expansion in that basis proves
\eqref{MixedRIISL} at the boundary. The type-I boundary recurrence follows
by duality.

Using \eqref{BOSL}, the recurrence coefficients, including the boundary
coefficients just defined, are
\begin{align}
	\label{MixedRecurrenceSL}
	b^{k}_{n}&
		={\int_\Delta x\, A_{n+{k}}(x) B_{n}(x) \d w(x)}, && k\in\{-p,\ldots,-1,0,1,\ldots,q\},
&& n\in\N_0.
\end{align}

With this notation, \eqref{MixedRIISL} and \eqref{MixedRISL} take the matrix
form
\begin{align*}
	T\begin{bNiceMatrix}
		B_0(x)\\B_1(x)\\
		B_2(x)\\
		 \Vdots
	\end{bNiceMatrix}&=x\begin{bNiceMatrix}
		B_0(x)\\B_1(x)\\
		B_2(x)\\
		\Vdots
	\end{bNiceMatrix},
	& &
	T^{\top}\begin{bNiceMatrix}
		A_0(x)\\
		A_1(x)\\
		A_2(x)\\
		\Vdots
	\end{bNiceMatrix}
	=x\begin{bNiceMatrix}
		A_0(x)\\
		A_1(x)\\
		A_2(x)\\
		\Vdots
	\end{bNiceMatrix}.
\end{align*}
where the banded matrix is
\begin{align}
	\label{T}
	T\coloneq
	\begin{bNiceMatrix}
		b^{0}_0 & b^{1}_0 & 	\Cdots[shorten-start=2pt,shorten-end=2pt]&
	 & 1 & 0 &
		\Cdots[shorten-start=-0em,shorten-end=-10pt] & & \phantom{0}\\
		b^{-1}_1 & b^{0}_1 & b^1_1 & & & &
		\Ddots[shorten-start=2pt,shorten-end=-15pt] & & \phantom{0}\\
		& b^{-1}_{2} & b^{0}_{2} &
		\Ddots[shorten-start=-0.10em,shorten-end=-20pt] & & &
		\Ddots[shorten-start=-0.10em,shorten-end=-17pt] & &\\
		\Vdots[shorten-start=2pt,shorten-end=2pt] & &
		\Ddots[shorten-start=-0.10em,shorten-end=-25pt] &
		\Ddots[shorten-start=5pt,shorten-end=-12pt] & & & & & \phantom{0}\\
		b^{-p}_p & & & & & & & & \phantom{1}\\
		0 & \Ddots[shorten-start=-0.10em,shorten-end=-25pt] & & & & & & & \phantom{0}\\
		&&&&&&&&\\
		\Vdots[shorten-start=-0.10em,shorten-end=-10pt] & &
		\Ddots[shorten-start=-0.15em,shorten-end=-10pt] & & & & & &
		\phantom{b^1_{m-2}}\\[4pt]
		\phantom{0} & \phantom{0} & & \phantom{0} &
		\phantom{b^{-p}_{m-1}} & \phantom{0} & &
		\phantom{b^{-1}_{m-1}} & \phantom{b^{0}_{m-1}}
	\end{bNiceMatrix}.
\end{align}
Notice that the normalization conditions \eqref{NormIISL}, \eqref{NormISL} imply that
\(b^q_n=1\),\(n\in\N_0\).

Equations~\eqref{BOSL} and \eqref{MixedRecurrenceSL} give the integral representation
\begin{align}
	\label{TIR}
		(T^n)_{i,j}=\int_{\Delta} x^n B_{i-1}(x)A_{j-1}(x)\,\d w(x),
		&& n\in\N_0, && i,j\in\N.
\end{align}

Its leading \(m\times m\) principal truncation is
\begin{align}
	\label{Tfinite}
	T_m\coloneq
	\begin{bNiceMatrix}
		b^{0}_0 & b^{1}_0 & &\Cdots &1&0&\Cdots&&0\\
		b^{-1}_1 & b^{0}_1 & b^1_1 &&&&\Ddots&&\Vdots\\
		&b^{-1}_{2}&b^{0}_{2}&\Ddots&&&\Ddots&&\\ \Vdots&&\Ddots&\Ddots&&&&&0\\
		b^{-p}_p&&&&&&&&1\\
		0&\Ddots&&&&&&&\Vdots\\
		&&&&&&&&\\
		\Vdots&&\Ddots&&&&&&b^1_{m-2}\\[4pt]
		0&\Cdots&&0&b^{-p}_{m-1}&\Cdots&&b^{-1}_{m-1}&b^{0}_{m-1}
	\end{bNiceMatrix},\quad m\in\N.
\end{align}

\subsection{The step line and the matrix of moments for a general matrix of measures}
We now pass to the moment matrix of a general matrix of measures. The
Gauss--Borel factorization recovers mixed orthogonality on the step line and
expresses the recurrence matrix in terms of triangular factors
\cite{BFM1}. Bidiagonal factorizations induced by elementary Christoffel
perturbations were studied in \cite{banda_bidiagonal}; below we adapt that
construction to the normalization used here.

Consider a general matrix of measures
\begin{align*}
	\d	\mu=\begin{bNiceMatrix}
 \d\mu_{1,1}&\Cdots &\d\mu_{1,p}\\
 \Vdots & & \Vdots\\
 \d	\mu_{q,1}&\Cdots &\d\mu_{q,p}
	\end{bNiceMatrix},
\end{align*}
where the measures \(\mu_{i,j}\) are supported on the interval
\(\Delta\subseteq\mathbb R\).
For \(r\in\mathbb N\), define the matrix of monomials by
\begin{align*}
X_{[r]}(x) =
\begin{bNiceMatrix}
	I_r \\
	xI_r \\
	x^2 I_r \\
	\Vdots
\end{bNiceMatrix} ,
 \end{align*}
where \( I_r \) is the usual identity matrix of order \( r \).
The moment matrix is
\begin{align*}
\mathscr{M}(\mu) := \int_{\Delta} X_{[q]}(x) \, \mathrm{d}\mu(x) \, X_{[p]}^\top(x).
 \end{align*}
When the measure is clear from context, we simply write \(\mathscr M\).

\subsection{Gauss--Borel factorization}
If all the leading principal submatrices \(\mathscr{M}^{[k]}\) are nonsingular, then the Gauss--Borel factorization exists:
\begin{align*}
\mathscr{M} = \mathscr{L}^{-1} \mathscr{U}^{-1},
 \end{align*}
where \(\mathscr{L}\) is a nonsingular lower unitriangular semi-infinite matrix and \(\mathscr{U}\) is a nonsingular upper triangular matrix. When necessary, we will denote these triangular matrices as \(\mathscr{L}(\mu)\) and \(\mathscr{U}(\mu)\), indicating the measure~\(\mu\) from which they are constructed.
All semi-infinite identities below are understood coefficientwise, or~equivalently through their compatible finite leading truncations; bandedness
ensures that every displayed matrix entry involves only finitely many terms.

\subsection{Mixed multiple orthogonal polynomials on the step-line}

Associated with the Gauss--Borel factorization, consider the following matrices of polynomials:
\begin{align*}
	B(x) = \mathscr{L} X_{[q]}(x), && A(x) = X_{[p]}^\top(x) \mathscr{U}.
 \end{align*}
We represent these matrices in terms of their polynomial entries as follows:
\begin{align*}
	B &= \begin{bNiceMatrix}
 B^{(1)}_0 & \Cdots & B^{(q)}_0 \\
 B^{(1)}_1 & \Cdots & B^{(q)}_1 \\
 B^{(1)}_2 & \Cdots & B^{(q)}_2 \\
 \Vdots[shorten-end=-0pt] & & \Vdots[shorten-end=-0pt]
	\end{bNiceMatrix}, &&
	A = \left[\begin{NiceMatrix}
 A^{(1)}_0 & A^{(1)}_1 & A^{(1)}_2 & \Cdots \\
 \Vdots & \Vdots & \Vdots & \\
 A^{(p)}_0 & A^{(p)}_1 & A^{(p)}_2 & \Cdots
	\end{NiceMatrix}\right].
 \end{align*}
We have the following relations:
\begin{align*}
\int_{\Delta} B(x) \, \mathrm{d}\mu(x) \, A(x) = I,
 \end{align*}
whose entries are given by the biorthogonality relations:
\begin{align*}
\int_{\Delta} \sum_{b=1}^q \sum_{a=1}^p B^{(b)}_n(x) \, \mathrm{d}\mu_{b,a}(x) \, A^{(a)}_m(x) = \delta_{n,m}.
 \end{align*}
The Gauss--Borel factorization also gives
\begin{align*}
	\int_\Delta B(x) \, \mathrm{d}\mu(x) X_{[p]}^\top(x) &= \mathscr{U}^{-1},
&&
\int_\Delta X_{[q]}(x) \, \mathrm{d}\mu(x) A(x) = \mathscr{L}^{-1},
\end{align*}
which, entrywise, give the following mixed multiple orthogonality relations
on the step line:
\begin{align*}
	\int_\Delta x^l \sum_{i=1}^p \mathrm{d}\mu_{j,i}(x) A_n^{(i)}(x) &= 0,
&
& j \in \{1, \ldots, q\}, && l \in \left\{0, \ldots, \left\lceil\frac{n-j+1}{q}\right\rceil-1\right\},
 \\
	\int_\Delta \sum_{j=1}^q B_n^{(j)}(x) \mathrm{d}\mu_{j,i}(x) x^l &= 0,
& 
& i \in \{1, \ldots, p\}, && l \in \left\{0, \ldots, \left\lceil\frac{n-i+1}{p}\right\rceil-1\right\}.
\end{align*}
The component \(A_n^{(i)}\) corresponds to
\(A_{\vec n,\vec m}^{(i)}\), where
\(\vec n=(n_1,\ldots,n_p)\)
and
\(\vec m=(m_1,\ldots,m_q)\)
are defined by
\begin{align}
 \label{eq:stepA-ni}
n_i = \left\lceil \frac{|\vec{n}|+1-i}{p} \right\rceil = \left\lceil \frac{n+2-i}{p} \right\rceil
&&
\text{and}
&&
m_j= \left\lceil \frac{|\vec{m}|+1-j}{q} \right\rceil = \left\lceil \frac{n+1-j}{q} \right\rceil
 .
 \end{align}
Lemma 2.1 of \cite{BFM1} gives \(|\vec n|=n+1\) and \(|\vec m|=n\).
Similarly, \(B_n^{(j)}\) corresponds to
\(B_{\vec n,\vec m}^{(j)}\), with
\(\vec n=(n_1,\ldots,n_p)\)
and
\(\vec m=(m_1,m_2,\ldots,m_q)\)
given by
\begin{align}
 \label{eq:n_idos}
n_i = \left\lceil \frac{n+1-i}{p} \right\rceil
&&
\text{and}
&&
m_j = \left\lceil \frac{n+2-j}{q} \right\rceil .
 \end{align}
In this case \(|\vec m|=|\vec n|+1\).
\subsection{Banded recurrence matrix}

For \(r\in\mathbb N\), define the block shift matrix by
\begin{align*}
\Lambda_{[r]} \coloneq
\left[
\begin{NiceMatrix}[cell-space-limits=2pt]
	0_r & I_r & 0_r & \Cdots \\
	0_r & 0_r & I_r & \Ddots \\
	0_r & 0_r & 0_r & \Ddots \\
	\Vdots[shorten-end=5pt] & \Ddots[shorten-end=5pt] & \Ddots[shorten-end=9pt] & \Ddots[shorten-end=12pt]
\end{NiceMatrix}
\right],
 \end{align*}
For \(r=1\), write \(\Lambda_{[1]}=\Lambda\). Then
\(\Lambda_{[r]}=\Lambda^r\), and
\begin{align*}
\Lambda_{[r]}X_{[r]}(x) = x X_{[r]}(x).
 \end{align*}

The moment matrix \(\mathscr M\) satisfies the Hankel-type symmetry
\begin{align*}
\Lambda_{[q]}\mathscr{M} = \mathscr{M}\Lambda_{[p]}^\top.
 \end{align*}
This relation and the Gauss--Borel factorization give
\begin{align}
\label{eq:banded_recurrence_matrix}
	T = \mathscr{L} \Lambda_{[q]} \mathscr{L}^{-1} = \mathscr{U}^{-1} \Lambda_{[p]}^\top \mathscr{U} ,
\end{align}
Thus \(T\) is a \((p,q)\)-banded matrix with \(p\) subdiagonals and \(q\)
superdiagonals. Moreover,
\begin{align*}
T B(x) = x B(x), && A(x) T = x A(x),
 \end{align*}
so \(B\) and \(A\) are, respectively, right and left eigenvectors of \(T\).
These identities encode the recurrence relations; hence \(T\) is the
recurrence matrix.

\subsection{Elementary Christoffel perturbations}
Define the \(r\times r\) polynomial matrix
\begin{align*}
\mathfrak{X}_{[r]}(x) =
\begin{bNiceMatrix}
	0 & 1 & 0 & \Cdots & 0 & 0 \\
	0 & 0 & 1 & \Ddots & & 0 \\
	\Vdots & \Vdots & \Ddots[shorten-end=10pt] & \Ddots & & \Vdots \\
	0 & 0 & & & 1 & 0 \\
	0 & 0 & \Cdots & & 0 & 1 \\
	x & 0 & \Cdots & & 0 & 0
\end{bNiceMatrix}.
 \end{align*}
It satisfies \(\mathfrak X_{[r]}^r=xI_r\) and
\(X_{[r]}\mathfrak{X}_{[r]} = \Lambda X_{[r]} \).

For \(k\in\mathbb N_0\), define the transformed matrices of measures by
\begin{align} \label{eq:Christoffel}
	\mathrm{d}\mu_L^{(k)} & \coloneq \mathrm{d}\mu \left(\mathfrak{X}_{[p]}^k\right)^\top, &
	\mathrm{d}\mu_R^{(k)} & \coloneq \mathfrak{X}_{[q]}^k \mathrm{d}\mu .
\end{align}
For \(k=0\), we have \(\mathrm{d}\mu_R^{(0)} = \mathrm{d}\mu_L^{(0)} = \mathrm{d}\mu\).
These are the Christoffel perturbations of \(\mathrm d\mu\). Their moment
matrices satisfy
\begin{align*}
	\mathscr{M}_L^{(k)} &\coloneq \mathscr{M}\left(\mu_L^{(k)}\right) =
	\mathscr{M}(\mu) \left(\Lambda^k\right)^\top, &
	\mathscr{M}_R^{(k)} &\coloneq \mathscr{M}\left(\mu_R^{(k)}\right) =
	\Lambda^k \mathscr{M}(\mu)
.
 \end{align*}
Assume that the perturbed moment matrices admit Gauss--Borel factorizations
for every left perturbation \(k=0,\ldots,p\) and every right perturbation
\(k=0,\ldots,q\):
 \begin{align*}
		\mathscr{M}_L^{(k)}& = \left(\mathscr{L}_L^{(k)}\right)^{-1}\left(\mathscr{U}_L^{(k)}\right)^{-1}, &
		\mathscr{M}_R^{(k)}& = \left(\mathscr{L}_R^{(k)}\right)^{-1}\left(\mathscr{U}_R^{(k)}\right)^{-1}
 .
 \end{align*}
In both cases, \(\mathscr{L}^{(k)}_L\) and \(\mathscr{L}^{(k)}_R\) are lower unitriangular matrices. \par

\section{Bidiagonal factorizations of the banded recurrence matrix}

We now consider the \(k\)-th left and right Christoffel perturbations.

The triangular factors of the perturbed moment matrices are connected by
elementary bidiagonal matrices.
More precisely,
\begin{pro}
The left and right transformations satisfy the following
factorization identities.
\begin{enumerate}[\rm i)]
 \item For the \(k\)-th left Christoffel transformation, there exist lower
 bidiagonal matrices with unit diagonal \(L_1,\ldots,L_k\) such that
\begin{align}\label{eq:left_Christoffel}
 \mathscr{U}_L^{-1}\left(\Lambda^k\right)^\top &= L_1 \cdots L_k \left(\mathscr{U}^{(k)}_L\right)^{-1}, &
 \mathscr{L}^{(k)}_L &= L_k^{-1} \cdots L_1^{-1} \mathscr{L}_L .
\end{align}
\item
 For the \(k\)-th right Christoffel transformation, there exist upper
 bidiagonal matrices with unit superdiagonal \(U_1,\ldots,U_k\) such that
\begin{align}\label{eq:right_Christoffel}
 \Lambda^k \mathscr{L}_R^{-1} &= \left(\mathscr{L}_R^{(k)}\right)^{-1} U_k \cdots U_1, &
 \mathscr{U}_R^{(k)} &= \mathscr{U}_R U_1^{-1} \cdots U_k^{-1} .
\end{align}
\end{enumerate}
\end{pro}

\begin{proof}
i)
From the Gauss--Borel factorizations of \(\mathscr{M}_L^{(1)}\) and \(\mathscr{M}\), we obtain:
\begin{align*}
\left(\mathscr{L}^{(1)}_L\right)^{-1}\left( \mathscr{U}^{(1)}_L\right)^{-1} = \mathscr{L}_L^{-1} \mathscr{U}_L^{-1}\Lambda^\top.
 \end{align*}
The matrix \(\mathscr{U}_L^{-1}\Lambda^\top\) is upper Hessenberg, with
Gauss--Borel factorization
\begin{align*}
\mathscr{U}_L^{-1}\Lambda^\top = L_1 \left(\mathscr{U}^{(1)}_L\right)^{-1},
 \end{align*}
where \(L_1\) is a lower bidiagonal matrix with unit diagonal. Consequently:
\begin{align*}
\mathscr{L}_L^{(1)} = L_1^{-1} \mathscr{L}_L.
 \end{align*}
In general, since
\begin{align*}
\mathscr{M}_L^{(k+1)} = \mathscr{M}_L^{(k)}\Lambda^\top,
 \end{align*}
we obtain:
\begin{align*}
\begin{aligned}
	\left( \mathscr{U}_L^{(k)} \right)^{-1}\Lambda^\top &= L_{k+1} \left(\mathscr{U}^{(k+1)}_L\right)^{-1}, &
	\mathscr{L}^{(k+1)}_L &= L_{k+1}^{-1} \mathscr{L}_L^{(k)},
\end{aligned}
 \end{align*}
where \(L_{k+1}\) is a lower bidiagonal matrix with unit main diagonal. This proves \eqref{eq:left_Christoffel}.

ii)
 Consider the right Christoffel perturbations.
For the Gauss--Borel factorization of \(\mathscr{M}_R^{(1)}\), we have:
\begin{align*}
\mathscr{M}_R^{(1)} = \Lambda \mathscr{L}_R^{-1} \mathscr{U}_R^{-1}.
 \end{align*}
The matrix \(\Lambda \mathscr{L}^{-1}\) is a lower Hessenberg matrix whose Gauss--Borel factorization~is:
\begin{align*}
\Lambda \mathscr{L}_R^{-1} = \left(\mathscr{L}_R^{(1)}\right)^{-1} U_1,
 \end{align*}
where \(U_1\) is an upper bidiagonal matrix with unit superdiagonal. Consequently:
\begin{align*}
\mathscr{U}_R^{(1)} = \mathscr{U}_R U_1^{-1}.
 \end{align*}
In general, since
\begin{align*}
\mathscr{M}_R^{(k+1)} = \Lambda \mathscr{M}_R^{(k)},
 \end{align*}
we obtain:
\begin{align*}
	\Lambda \left( \mathscr{L}_R^{(k)} \right)^{-1} &= \left(\mathscr{L}_R^{(k+1)}\right)^{-1} U_{k+1}, &
	\mathscr{U}_R^{(k+1)} &= \mathscr{U}_R^{(k)} U_{k+1}^{-1}.
 \end{align*}
Consequently \eqref{eq:right_Christoffel} follows.
\end{proof}

Iterating the elementary Christoffel connections produces the bidiagonal
factorization of the recurrence matrix.
\begin{teo}
Suppose that the Gauss--Borel factorizations just specified exist for
\(k=0,\ldots,p\) on the left and \(k=0,\ldots,q\) on the right. Then the
banded recurrence matrix has the bidiagonal factorization
\begin{align}\label{eq:T_L_bidiagonal}
	T& = L_1 \cdots L_p U_q \cdots U_1.
\end{align}
\end{teo}

\begin{proof}
Equations \eqref{eq:banded_recurrence_matrix} and
\eqref{eq:left_Christoffel} give
\begin{align*}
	T&= \mathscr{U}_L^{-1} \left(\Lambda^p\right)^\top \mathscr{U}_L = L_1 \cdots L_p \left(\mathscr{U}^{(p)}_L\right)^{-1}\mathscr{U}_L,
 \end{align*}
Since \(\mathfrak X_{[p]}^p=xI_p\) and
\(\mathfrak X_{[q]}^q=xI_q\),
\begin{align*}
 \d\mu_L^{(p)}=x\,\d\mu=\d\mu_R^{(q)}.
\end{align*}
The uniqueness of the normalized Gauss--Borel factorization therefore gives
\(\mathscr U_L^{(p)}=\mathscr U_R^{(q)}\). Using
\eqref{eq:right_Christoffel} we~obtain:
\begin{align*}
\begin{aligned}
	\left(\mathscr{U}^{(q)}_R\right)^{-1}\mathscr{U}_L & = U_q \cdots U_1
\end{aligned}
 \end{align*}
which proves \eqref{eq:T_L_bidiagonal}.

\end{proof}
The proof yields the following corollary.
\begin{coro}
	The bidiagonal matrices satisfy
\begin{align*}
	L_{k} &= \left( \mathscr{U}_L^{(k-1)} \right)^{-1}\Lambda^\top\mathscr{U}^{(k)}_L =\mathscr{L}_L^{(k-1)}\left(\mathscr{L}^{(k)}_L\right)^{-1}, &
	U_k &= \mathscr{L}_R^{(k)}\Lambda \left( \mathscr{L}_R^{(k-1)} \right)^{-1} =\left(\mathscr{U}^{(k)}_R\right)^{-1}\mathscr{U}_R^{(k-1)},
 \end{align*}
and the corresponding entries are
\begin{align*}
	(L_k)_{n+1,n}&=\frac{(\mathscr U_L^{(k)})_{n,n}}{(\mathscr U_L^{(k-1)})_{n+1,n+1}}=(\mathscr L_L^{(k-1)})_{n+1,n}-(\mathscr L_L^{(k)})_{n+1,n},&
	 (U_k)_{n,n}&=\frac{(\mathscr U_R^{(k-1)})_{n,n}}{(\mathscr U_R^{(k)})_{n,n}}.
\end{align*}

\end{coro}

Let \(r(n,p)\) denote the remainder when \(n\) is divided by \(p\).
After \(k\) left or right Christoffel perturbations, denote the corresponding
mixed multiple orthogonal polynomials by \(A_L^{(k)}\) and \(A_R^{(k)}\),
respectively.

Given a polynomial \( P\) we denote its leading coefficient by \( \operatorname{LC}
(P) \).
Then, we can express the bidiagonal matrix entries as follows:
\begin{coro}\label{coro}
	In terms of the leading coefficients of the Christoffel-perturbed mixed multiple orthogonal polynomials, we have:
\begin{align*}
 (L_k)_{n+1,n}&=\frac{\operatorname{LC}\left((A_L^{(k)})^{(r(n,p)+1)}_n\right)}{\operatorname{LC}\left((A_L^{(k-1)})^{(r(n+1,p)+1)}_{n+1}\right)},&
 (U_k)_{n,n}&=\frac{\operatorname{LC}\left((A_R^{(k-1)})^{(r(n,p)+1)}_n\right)}{\operatorname{LC}\left((A_R^{(k)})^{(r(n,p)+1)}_{n}\right)}.
 \end{align*}
\end{coro}
\begin{proof}
Write \(n=ps+r\), with \(0\le r<p\). From
\(A=X_{[p]}^{\top}\mathscr U\), the entry
\((\mathscr U)_{n,n}\) multiplies the monomial \(x^s\) in component
\(r+1\) of \(A_n\). Upper triangularity implies that every other
contribution to that component has degree at most \(s-1\). Hence
\begin{align*}
 (\mathscr U)_{n,n}
 =\operatorname{LC}\bigl((A)^{(r(n,p)+1)}_n\bigr).
\end{align*}
The same identity holds for each left and right perturbed factorization.
Substitution into the two diagonal-entry ratios in the preceding corollary
gives the displayed formulas.
\end{proof}

Let \(T_L^{(k)}\), respectively
\(T_R^{(k)}\),
denote the recurrence matrix
after the \(k\)-th left, respectively right, Christoffel transformation.
Both~use the normalization
fixed above.
The preceding factorizations identify these transformations as Darboux
rearrangements.
\begin{pro}
The perturbed recurrence matrices are the Darboux transformations of the recurrence matrix:
\begin{align*}
 T^{(k)}_L&= L_{k+1}\cdots L_p U_q\cdots U_1L_1\cdots L_k,
 & k&\in\{1,\ldots,p\},\\
 T^{(k)}_R&= U_k\cdots U_1L_1\cdots L_pU_q\cdots U_{k+1},
 & k&\in\{1,\ldots,q\}.
	\end{align*}
\end{pro}

\begin{proof}
Equations~\eqref{eq:banded_recurrence_matrix},
\eqref{eq:left_Christoffel}, and \eqref{eq:right_Christoffel} give
	\begin{align*}
T^{(k)}_L &= \left(\mathscr{U}^{(k)}_L\right)^{-1} (\Lambda^p)^\top \mathscr{U}^{(k)}_L \\&=
L_{k+1} \cdots L_p \left(\mathscr{U}^{(p)}_L\right)^{-1} (\Lambda^{k})^\top \mathscr{U}^{(k)}_L\\
&= L_{k+1} \cdots L_p \left(\mathscr{U}^{(q)}_R\right)^{-1} (\Lambda^{k})^\top \mathscr{U}^{(k)}_L
\\&=L_{k+1} \cdots L_p U_q \cdots U_1 \left(\mathscr{U}^{(0)}_R\right)^{-1} (\Lambda^{k})^\top \mathscr{U}^{(k)}_L\\&= L_{k+1} \cdots L_p U_q \cdots U_1 \left(\mathscr{U}^{(0)}_L\right)^{-1} (\Lambda^{k})^\top \mathscr{U}^{(k)}_L\\&=L_{k+1} \cdots L_p U_q \cdots U_1 L_1 \cdots L_k
\end{align*}
Again, from \eqref{eq:banded_recurrence_matrix}, \eqref{eq:left_Christoffel} and \eqref{eq:right_Christoffel}
\begin{align*}
T^{(k)}_R &= \left(\mathscr{U}^{(k)}_R\right)^{-1} (\Lambda^p)^\top \mathscr{U}^{(k)}_R \\&=
U_{k} \cdots U_1 \left(\mathscr{U}^{(0)}_R \right)^{-1} (\Lambda^{p})^\top \mathscr{U}^{(k)}_R\\
&= U_{k} \cdots U_1 \left(\mathscr{U}^{(0)}_L \right)^{-1} (\Lambda^{p})^\top \mathscr{U}^{(k)}_R\\
&= U_{k} \cdots U_1 L_1 \cdots L_p \left(\mathscr{U}^{(p)}_L\right)^{-1} \mathscr{U}^{(k)}_R
\\&=U_{k} \cdots U_1 L_1 \cdots L_p \left(\mathscr{U}^{(q)}_R \right)^{-1} \mathscr{U}^{(k)}_R
\\&= U_{k} \cdots U_1 L_1 \cdots L_p U_q \cdots U_{k+1}
\end{align*}\end{proof}

\section{Mixed Piñeiro orthogonal polynomials}

Let
\(\vec{\alpha} = \begin{bNiceMatrix}
	\alpha_1 & \Cdots & \alpha_p ,
\end{bNiceMatrix}\),
\(\vec{\beta} = \begin{bNiceMatrix}
	\beta_1 & \Cdots & \beta_q
\end{bNiceMatrix} \),
and introduce a mixed Piñeiro matrix of measures of the~form
\begin{align*}
\d \mu_{\vec{\alpha}, \vec{\beta}} =
\begin{bNiceMatrix}
	x^{\beta_1} \\ \Vdots \\ x^{\beta_q}
\end{bNiceMatrix}
\begin{bNiceMatrix}
	x^{\alpha_1} & \Cdots & x^{\alpha_p}
\end{bNiceMatrix} \, \d x,
 \end{align*}
supported on \(\Delta=[0,1]\). We assume \(\alpha_i>-1\), \(\beta_j>-1\),
and
\begin{align*}
 \alpha_i+\beta_j>-1,
 && 1\le i\le p, && 1\le j\le q,
\end{align*}
so that every entry \(x^{\alpha_i+\beta_j}\,\d x\) is integrable. We also
assume
\(\alpha_i-\alpha_h\notin\mathbb Z\) for \(i\ne h\) and
\(\beta_j-\beta_s\notin\mathbb Z\) for \(j\ne s\). Under these
hypotheses the two systems are AT systems and the corresponding matrix of
measures is perfect; in particular, all admissible multi-indices are normal
\cite{Ulises-Sergio-Judith}.

For \(|\vec m|=|\vec n|-1\), write the type-I linear form as
\(
A_{\vec{n},\vec{m}} (x)= \sum_{i=1}^p A_{\vec{n},\vec{m}}^{(i)} (x) x^{\alpha_i}
\)
where \(\deg A_{\vec n,\vec m}^{(i)}\le n_i-1\).
It~satisfies
\begin{align*}
\int_0^1 A_{\vec n,\vec m}(x)x^{k+\beta_j}\,\d x=0,
&& 0\le k\le m_j-1,&& 1\le j\le q.
\end{align*}
Perfectness implies that these conditions determine the form up to a nonzero
constant.
Motivated by the contour method in \cite{BDFMW}, we obtain the following
direct representation.
\begin{pro}
Up to a nonzero constant, the type-I form has the contour representation
\begin{align}
\label{eq:contourPinheiro}
A_{\vec{n},\vec{m}} (x) = \int_{\Sigma} \frac{x^t \prod_{j=1}^q (t+\beta_j+1)_{m_j}}{\prod_{i=1}^p (\alpha_i-t)_{n_i}} \frac{dt}{ 2 \pi i }
\end{align}
where \(\Sigma\subset\mathbb C\) is a clockwise contour enclosing once every
point
\begin{align*}
 \alpha_i,\alpha_i+1,\ldots,\alpha_i+n_i-1,
 && 1\le i\le p,
\end{align*}
and no other pole of the integrand. For \(0<x<1\), we use
\(x^t=\exp(t\log x)\) with the real logarithm.
\end{pro}
For later sign checks, recall that our clockwise convention gives
\begin{align*}
 \frac{1}{2\pi\mathrm i}\int_\Sigma f(t)\,\d t
 =-\sum_{a\text{ inside }\Sigma}\operatorname{Res}_{t=a}f(t).
\end{align*}

\begin{proof}
By the residue theorem, the contour integral is a finite linear combination
of the functions \(x^{\alpha_i+k}\), \(0\le k<n_i\), and hence has the
required type-I form and component degree bounds. We may therefore compute
its Mellin transform term by term. Equivalently, for
\(s=\beta_j+h\), \(1\le h\le m_j\), we have
\begin{align*}
\int_0^1 A_{\vec n,\vec m}(x)x^{s-1}\,\d x
 = \int_\Sigma
 \frac{\prod_{r=1}^q(t+\beta_r+1)_{m_r}}
 {\prod_{i=1}^p(\alpha_i-t)_{n_i}}
 \frac{1}{t+s}\,\frac{\d t}{2\pi\mathrm i}.
\end{align*}
The numerator contains the factor
\(t+s=t+\beta_j+h\), which cancels the last denominator. Since
\(|\vec m|=|\vec n|-1\), the resulting rational function is
\(O(t^{-2})\) at infinity, and hence its residue at infinity is zero. After
cancellation its only finite poles are the points enclosed by \(\Sigma\).
The global residue theorem therefore gives
\begin{align*}
 \sum_{a\text{ inside }\Sigma}\operatorname{Res}_{t=a}f(t)
 =-\operatorname{Res}_{t=\infty}f(t)=0.
\end{align*}
The clockwise residue convention then makes the displayed integral vanish.
This proves all the required orthogonality conditions. The residue at
\(t=\alpha_i\) is nonzero whenever \(n_i\ge1\), under the standing
parameter assumptions. Perfectness then identifies the resulting nonzero
form, up to its common multiplicative
constant, with \(A_{\vec n,\vec m}\).
\end{proof}

Taking residues in the preceding contour representation produces the
hypergeometric form componentwise.
\begin{coro}
For every \(i\in\{1,\ldots,p\}\) with \(n_i\ge1\), the integral
representation gives the following explicit formula, up to a common nonzero
constant:
\begin{multline}
\label{PinheiroI}
A_{\vec{n},\vec{m}}^{(i)} (x)
=\frac{\prod_{j=1}^q (\alpha_i+\beta_j+1)_{m_j}}{(n_i-1)!
\prod_{j=1, j\ne i}^p(\alpha_j-\alpha_i)_{n_j}} \\
\times \pFq{p+q}{p+q-1}{-n_i+1,\,\{\alpha_i+\beta_k+m_k+1\}_{k=1}^q,\,\{\alpha_i-\alpha_k-n_k+1\}_{k=1,k\neq i}^p}{\{\alpha_i+\beta_k+1\}_{k=1}^q,\,\{\alpha_i-\alpha_k+1\}_{k=1,k\neq i}^p}{x} .
 \end{multline}
The individual coefficients are
 \begin{align}
 \label{eq:kcoeficientAn}
 A_{\vec{n},\vec{m}}^{(i)}[k] = \frac{\prod_{j=1}^q (\alpha_i+k+\beta_j+1)_{m_j}}{\prod_{j=1, j\ne i}^p (\alpha_j-\alpha_i-k)_{n_j} k! (n_i-k-1)! (-1)^{k}} ,
 \end{align}
and the leading coefficient is
\begin{align*}
\operatorname{LC} \left(A_{\vec{n},\vec{m}}^{(i)}\right)= A_{\vec{n},\vec{m}}^{(i)}[n_i-1] = \frac{\prod_{j=1}^q (\alpha_i+n_i+\beta_j)_{m_j}}{\prod_{j=1, j\ne i}^p (\alpha_j-\alpha_i-n_i+1)_{n_j} (n_i-1)! (-1)^{n_i-1}} .
\end{align*}
\end{coro}

\begin{proof}
Write
\(
\displaystyle
 A_{\vec{n},\vec{m}}^{(i)} (x)=\sum_{k=0}^{n_i-1} A_{\vec{n},\vec{m}}^{(i)}[k]x^k \).
The clockwise residues at \(t=\alpha_i+k\) give
\eqref{eq:kcoeficientAn}; consequently,
 \begin{align}
 \label{eq:prep}
 A_{\vec{n},\vec{m}}^{(i)} (x)=\sum_{k=0}^{n_i-1} \frac{\prod_{j=1}^q (\alpha_i+k+\beta_j+1)_{m_j}}{\prod_{j=1, j\ne i}^p (\alpha_j-\alpha_i-k)_{n_j} k! (n_i-k-1)! (-1)^{k}}
x^k
 \end{align}
Using the identities
 \begin{align*}
 \frac{1}{(n_i-k-1)! (-1)^k} &= \frac{(-n_i+1)_k}{(n_i-1)!},\\
(\alpha_i+k+\beta_j+1)_{m_j} &= \frac{(\alpha_i+\beta_j+1)_{m_j+k}}{(\alpha_i+\beta_j+1)_{k}}
= \frac{(\alpha_i+\beta_j+1)_{m_j}(\alpha_i+\beta_j+1+m_j)_{k}}{(\alpha_i+\beta_j+1)_{k}} ,\\
\frac{1}{(\alpha_j-\alpha_i-k)_{n_j}} &= \frac{(-1)^{n_j}}{(-\alpha_j+\alpha_i+k-n_j+1)_{n_j}} = \frac{(-1)^{n_j}(-\alpha_j+\alpha_i-n_j+1)_{k} }{(-\alpha_j+\alpha_i-n_j+1)_{n_j+k}} \\
&= \frac{(-1)^{n_j}(-\alpha_j+\alpha_i-n_j+1)_{k} }{(-\alpha_j+\alpha_i-n_j+1)_{n_j} (-\alpha_j+\alpha_i+1)_{k}} = \frac{(-\alpha_j+\alpha_i-n_j+1)_{k} }{(\alpha_j-\alpha_i)_{n_j} (-\alpha_j+\alpha_i+1)_{k}} .
 \end{align*}
Equation \eqref{eq:prep} becomes
 \begin{align*}
 A_{\vec{n},\vec{m}}^{(i)} (x)=
 \frac{\prod_{j=1}^q (\alpha_i+\beta_j+1)_{m_j}}
 {(n_i-1)!\prod_{j\ne i}(\alpha_j-\alpha_i)_{n_j}}\sum_{k=0}^{\infty}
 \frac{(-n_i+1)_k
 \prod_{j=1}^q(\alpha_i+\beta_j+1+m_j)_k
 \prod_{j\ne i}(\alpha_i-\alpha_j-n_j+1)_k}
 {\prod_{j=1}^q(\alpha_i+\beta_j+1)_k
 \prod_{j\ne i}(\alpha_i-\alpha_j+1)_k}
 \frac{x^k}{k!},
 \end{align*}
which is \eqref{PinheiroI}.
 \end{proof}

The same contour representation also gives all mixed moments of the type-I
form.
 \begin{pro}
For \(s\in\{1,\ldots,q\}\), these moments are
\begin{align}
 \label{eq:momentosgeneralizadosAn}
 \int_0^1 x^{k} x^{\beta_s} A_{\vec{n},\vec{m}} (x)\,\d x= \frac{ \prod_{j=1}^q (\beta_j-\beta_{s} -k )_{m_j}}{\prod_{i=1}^p (\alpha_i+ \beta_{s} +k +1)_{n_i}},\qquad k\in\mathbb N_0.
 \end{align}
\end{pro}
\begin{proof}
Using \eqref{eq:contourPinheiro}, we obtain
 \begin{align*}
\int_0^1 x^{k} x^{ \beta_{s} } A_{\vec{n},\vec{m}} (x) =
 \int_{\Sigma} \frac{ \prod_{j=1}^q (t+\beta_j+1)_{ m_j } }{ \prod_{i=1}^p (\alpha_i-t)_{n_i } }
\frac{1}{k +t +\beta_{s} +1}\frac{\d t}{ 2 \pi \mathrm i } .
 \end{align*}
The rational integrand is \(O(t^{-2})\) at infinity. If \(k<m_s\), the
factor \(t+\beta_s+k+1\) cancels and both sides vanish. If \(k\ge m_s\),
evaluation from the exterior of the clockwise contour leaves the single pole
\(t=-(k+\beta_s+1)\); its~residue is the quotient displayed above.
 \end{proof}

For \(|\vec m|=|\vec n|+1\), write the type-II linear form as
\(\displaystyle
B_{\vec{n}, \vec{m}} (x) = \sum_{j=1}^q B^{(j)}_{\vec{n}, \vec{m}} (x) x^{\beta_j} \).
It satisfies
\begin{align*}
\int_0^1B_{\vec n,\vec m}(x)x^{k+\alpha_i}\,\d x=0,
&& 0\le k\le n_i-1, && 1\le i\le p.
\end{align*}
Interchanging the two parameter families gives the dual contour
representation.

\begin{pro}
Up to a nonzero constant, the type-II form has the contour representation
\begin{align}
\label{eq:contourPinheiroB}
B_{\vec{n},\vec{m}} (x) = \int_{\Sigma} \frac{x^t \prod_{i=1}^p (t+\alpha_i+1)_{n_i}}{\prod_{j=1}^q (\beta_j-t)_{m_j}} \frac{dt}{ 2 \pi i }
\end{align}
where \(\Sigma\) is a clockwise contour enclosing once every pole
\(\beta_j,\beta_j+1,\ldots,\beta_j+m_j-1\).
\end{pro}
\begin{proof}
This is the type-I argument above with
\((\vec\alpha,\vec n)\) and \((\vec\beta,\vec m)\) interchanged.
\end{proof}
Evaluating the residues of the dual contour integral yields the type-II
components explicitly.
\begin{coro}
For every \(j\in\{1,\ldots,q\}\) with \(m_j\ge1\), the integral
representation gives the following explicit formula, up to a common nonzero
constant:
 \begin{multline}
	\label{PinheiroII}
		B^{(j)}_{\vec{n},\vec{m}}(x)=\dfrac{\prod_{i=1}^p(\alpha_i+\beta_j+1)_{n_i}}{(m_j-1)!\prod_{i=1,i\neq j}^q({\beta_i-\beta_j})_{m_i}}\\
	\times\pFq{p+q}{p+q-1}{-m_j+1,\,\{\alpha_k+\beta_j+n_k+1\}_{k=1}^p,\,\{\beta_j-\beta_k-m_k+1\}_{k=1,k\neq j}^q}{\{\alpha_k+\beta_j+1\}_{k=1}^p,\,\{\beta_j-\beta_k+1\}_{k=1,k\neq j}^q}{x}.
\end{multline}
The individual coefficients are
 \begin{align}
 \label{eq:kcoeficientBn}
 B_{\vec{n},\vec{m}}^{(j)}[k] & = \frac{\prod_{i=1}^p (\alpha_i+k+\beta_j+1)_{n_i}}{\prod_{i=1, i\ne j}^q (\beta_i-\beta_j-k)_{m_i} k! (m_j-k-1)! (-1)^{k}} , \\
\label{eq:LCB}
\operatorname{LC} \left(B_{\vec{n},\vec{m}}^{(j)}\right) &= B_{\vec{n},\vec{m}}^{(j)}[m_j-1] = \frac{\prod_{i=1}^p (\beta_j+m_j+\alpha_i)_{n_i}}{\prod_{i=1, i\ne j}^q (\beta_i-\beta_j-m_j+1)_{m_i} (m_j-1)! (-1)^{m_j-1}} .
\end{align}
\end{coro}
\begin{proof}
Interchange \((\vec\alpha,\vec n)\) and \((\vec\beta,\vec m)\) in the
residue calculation used for the type-I components.
\end{proof}
The dual contour representation likewise determines all mixed moments of the
type-II form.
\begin{pro}
For \(s\in\{1,\ldots,p\}\), these moments are
\begin{align}
 \label{eq:momentosgeneralizadosBn}
 \int_0^1 x^{k} x^{\alpha_s} B_{\vec{n},\vec{m}} (x)\,\d x= \frac{ \prod_{i=1}^p (\alpha_i-\alpha_{s} -k )_{n_i}}{\prod_{j=1}^q (\beta_j+ \alpha_{s} +k +1)_{m_j}}, && k\in\mathbb N_0.
\end{align}
\end{pro}
\begin{proof}
This is the type-I moment calculation with the two parameter systems
interchanged.
\end{proof}
 \subsection{Step line}

For each \(n\in\mathbb N_0\), let
\(A_n:=A_{\vec n,\vec m}\), where \(|\vec n|=n+1\) and \(|\vec m|=n\),
with the contour representation \eqref{eq:contourPinheiro}. The components of
\(\vec n=(n_{1,n},n_{2,n},\ldots,n_{p,n})\) and
\(\vec m=(m_{1,n},m_{2,n},\ldots,m_{q,n})\)
are
 \begin{align}
 \label{eq:PinheiroA-ni}
n_{i,n} = \left\lceil \frac{|\vec{n}|+1-i}{p} \right\rceil = \left\lceil \frac{n+2-i}{p} \right\rceil 
&&
\text{and}
&&
m_{j,n} = \left\lceil \frac{|\vec{m}|+1-j}{q} \right\rceil = \left\lceil \frac{n+1-j}{q} \right\rceil .
 \end{align}

Write \(n=q\nu+\rho\), where \(0\leq\rho<q\), and set
\(r(n,q):=\rho\). The quotient \(\nu\) will not be used below.
We first fix the normalization of the type-I form on the step line.
 \begin{pro}
Impose the normalization
 \begin{align}
 \label{eq:normalization}
 \int_0^1 x^{m_{r(n,q)+1,n}} x^{ \beta_{ r(n,q)+1} } A_n(x)\,\d x = 1 .
 \end{align}
Then
\begin{align*}
A_{n}(x)=
\frac{
\prod_{i=1}^p (
\alpha_i+ \beta_{r(n,q)+1} + m_{r(n,q)+1,n}+1
)_{ n_{i,n} }
}{
\prod_{j=1}^q (
\beta_j-\beta_{r(n,q)+1}-m_{r(n,q)+1,n}
)_{m_{j,n}}
} \int_{\Sigma}
\frac{x^t \prod_{j=1}^q (t+\beta_j+1)_{ m_{j,n} } }{\prod_{i=1}^p (\alpha_i-t)_{n_{i,n} } } \frac{dt}{ 2 \pi i } ,
\end{align*}
 where \( n_{i,n}\) and \( m_{j,n}\) are given by \eqref{eq:PinheiroA-ni}.
\end{pro}
\begin{proof}
Set \(s=r(n,q)+1\) and \(k=m_{s,n}\) in
\eqref{eq:momentosgeneralizadosAn}. The displayed prefactor is exactly the
reciprocal of that moment, so the normalized moment equals \(1\).
\end{proof}

Substitution of the step-line multi-indices into the type-I formula gives the
normalized components.
 \begin{coro}
 For every \(i\in\{1,\ldots,p\}\) with \(n_{i,n}\ge1\), we obtain the explicit representation of the mixed polynomial~\( A_{n}^{(i)}(x) \), with the normalization condition \eqref{eq:normalization}
 \begin{multline*}
 A_{n}^{(i)}(x) = \frac{\prod_{j=1}^p (\alpha_j+ \beta_{r(n,q)+1} + m_{r(n,q)+1,n}+1)_{n_{j,n}}}{\prod_{j=1}^q (\beta_j-\beta_{r(n,q)+1}-m_{r(n,q)+1,n})_{m_{j,n}}} \frac{\prod_{j=1}^q (\alpha_i+\beta_j+1)_{m_{j,n}}}{(n_{i,n}-1)! \prod_{j=1, j \ne i}^p (\alpha_j-\alpha_i)_{n_{j,n}}} \\
\times
 \pFq{p+q}{p+q-1}{-n_{i,n}+1,\,\{\alpha_i+\beta_k+m_{k,n}+1\}_{k=1}^q,\,\{\alpha_i-\alpha_k-n_{k,n}+1\}_{k=1,k\neq i}^p}{\{\alpha_i+\beta_k+1\}_{k=1}^q,\,\{\alpha_i-\alpha_k+1\}_{k=1,k\neq i}^p}{x} .
\end{multline*}
If \(n_{i,n}=0\), then \(A_n^{(i)}\equiv0\). Here
\(n_{i,n}\) and \(m_{j,n}\) are given by \eqref{eq:PinheiroA-ni}.
\end{coro}
Expanding the normalized components gives their scalar coefficients.
\begin{coro}
For \(n_{i,n}\ge1\),
\(
A_{n}^{(i)}(x) = \sum_{k=0}^{n_{i,n}-1} A_{n}^{(i)}[k] x^k \)
where
\begin{align}
\label{PinheiroCoefiStepAn}
A_{n}^{(i)}[k] =
\frac{\prod_{j=1}^p (\alpha_j+ \beta_{r(n,q)+1} + m_{r(n,q)+1,n}+1)_{n_{j,n}}}{\prod_{j=1}^q (\beta_j-\beta_{r(n,q)+1}-m_{r(n,q)+1,n})_{m_{j,n}}} \frac{\prod_{j=1}^q (\alpha_i+k+\beta_j+1)_{m_{j,n}}}{\prod_{j=1, j\ne i}^p (\alpha_j-\alpha_i-k)_{n_{j,n}} k! (n_{i,n}-k-1)! (-1)^{k}}
\end{align}
and
\( n_{i,n}\) and \( m_{j,n}\)
given by \eqref{eq:PinheiroA-ni}.
\end{coro}
The last coefficient in the preceding expansion gives the corresponding
leading coefficient.
\begin{coro}

For \(n_{i,n}\ge1\),
 \begin{align*}
\operatorname{LC}
A_{n}^{(i)}
= (-1)^{n_{i,n}-1}\frac{\prod_{j=1}^p (\alpha_{j}+ \beta_{r(n,q)+1} + m_{r(n,q)+1,n}+1)_{n_{j,n}}}{\prod_{j=1}^q (\beta_j-\beta_{r(n,q)+1}-m_{r(n,q)+1,n})_{m_{j,n}}} \frac{\prod_{j=1}^q (\alpha_i+n_{i,n}+\beta_j)_{m_{j,n}}}{\prod_{j=1, j\ne i}^p (\alpha_j-\alpha_i-n_{i,n}+1)_{n_{j,n}} (n_{i,n}-1)! }
\end{align*}
where \( n_{i,n}\) and \( m_{j,n}\) are given by \eqref{eq:PinheiroA-ni}.
\end{coro}
The same normalization also yields a closed expression for every mixed
moment.
\begin{coro}
For \(s\in\{1,\ldots,q\}\), the moments of \(A_n\) are
\begin{align}
\label{eq:normAn}
 \int_0^1 x^{k} x^{\beta_s} A_{n} (x)\,\d x= \frac{\prod_{j=1}^p (\alpha_j+ \beta_{r(n,q)+1} + m_{r(n,q)+1,n}+1)_{n_{j,n}}}{\prod_{j=1}^q (\beta_j-\beta_{r(n,q)+1}-m_{r(n,q)+1,n})_{m_{j,n}}} \frac{ \prod_{j=1}^q (\beta_j-\beta_{s} -k )_{m_{j,n}}}{\prod_{i=1}^p (\alpha_i+ \beta_{s} +k +1)_{n_{i,n}}}, \ k\in\mathbb N_0,
 \end{align}
where \( n_{i,n}\) and \( m_{j,n}\) are given by \eqref{eq:PinheiroA-ni}.
\end{coro}
Analogously, for each \(n\in\mathbb N_0\), set
\(B_n(x):=B_{\vec n,\vec m}(x)\), where \(|\vec n|=n\) and
\(|\vec m|=n+1\),
\begin{align}
B_{\vec{n}, \vec{m}}(x) = \sum_{j=1}^q B^{(j)}_{\vec{n}, \vec{m} } (x)x^{\beta_j}
 .
\end{align}
Here \(\vec n=(n_{1,n},n_{2,n},\ldots,n_{p,n})\)
and
\(\vec m=(m_{1,n},m_{2,n},\ldots,m_{q,n})\)
are
defined by
 \begin{align}
 \label{eq:overlineniB}
n_{i,n} =\left\lceil \frac{|\vec{n}|+1-i}{p}\right\rceil = \left\lceil \frac{n+1-i}{p} \right\rceil
&&
\text{and}
&&
m_{j,n} = \left\lceil \frac{|\vec{m}|+1-j}{q} \right\rceil = \left\lceil \frac{n+2-j}{q} \right\rceil .
 \end{align}
Again write \(n=q\nu+\rho\), where \(0\leq\rho<q\), and set
\(r(n,q):=\rho\).
We impose the normalization
 \begin{align}
 \label{eq:normalizationB}
 \operatorname{LC} B^{(r(n,q)+1)}_n = 1 .
 \end{align}

This monic condition fixes the multiplicative constant in the type-II
contour representation.
 \begin{pro}
 Imposing this normalization condition gives
\begin{multline*}
B_{n}(x) =\frac{\prod_{j=1, j \ne r(n,q)+1}^q (\beta_j- \beta_{r(n,q)+1} - m_{r(n,q)+1,n}+1)_{m_{j,n}}
}{\prod_{i=1}^p (\beta_{r(n,q)+1}+m_{r(n,q)+1,n} + \alpha_i)_{n_{i,n}}} \\
\times
(m_{r(n,q)+1,n}-1)! (-1)^{m_{r(n,q)+1,n} -1}
\int_{\Sigma} \frac{x^t \prod_{i=1}^p (t+\alpha_i+1)_{n_{i,n}}}{\prod_{j=1}^q (\beta_j-t)_{m_{j,n}}} \frac{dt}{ 2 \pi i }
\end{multline*}
where \(n_{i,n}\) and \(m_{j,n}\) are given by \eqref{eq:overlineniB}.
 \end{pro}
 \begin{proof}
To calculate \( B_n \) we can use the integral representation obtained in \eqref{eq:contourPinheiroB} up to a constant \(  \kappa \). In this case, taking into account \eqref{eq:LCB}, we have
\begin{align*}
\operatorname{LC} B_n^{(r(n,q)+1)} =  \frac{ \kappa \, \prod_{i=1}^p (\beta_{r(n,q)+1}+m_{r(n,q)+1,n}+\alpha_i)_{n_{i,n}}}{\prod_{j=1, j\ne {(r(n,q)+1)}}^q (\beta_j-\beta_{r(n,q)+1}-m_{r(n,q)+1,n}+1)_{m_{j,n}} (m_{r(n,q)+1,n}-1)! (-1)^{m_{r(n,q)+1,n} -1}} .
\end{align*}
To determine this constant \( \kappa \) we impose the normalization \( \operatorname{LC} B_n^{(r(n,q)+1)} = 1\), and we obtain the result.
\end{proof}
The preceding normalization constant gives every type-II component on the
step line.
\begin{coro}
Put \(j_0=r(n,q)+1\) and
\begin{align*}
C_n=
(-1)^{m_{j_0,n}-1} (m_{j_0,n}-1)!
 \, \frac{
 \prod_{h\ne j_0}(\beta_h-\beta_{j_0}-m_{j_0,n}+1)_{m_{h,n}}
 }
 {
 \prod_{i=1}^p
 (\beta_{j_0}+m_{j_0,n}+\alpha_i)_{n_{i,n}}}.
\end{align*}
With the normalization \eqref{eq:normalizationB}, the components are
given, for \(m_{j,n}\ge1\), by
\begin{multline*}
B_n^{(j)}(x)=C_n
\frac{\prod_{i=1}^p(\alpha_i+\beta_j+1)_{n_{i,n}}}
{(m_{j,n}-1)!\prod_{h\ne j}(\beta_h-\beta_j)_{m_{h,n}}}\\
\times\pFq{p+q}{p+q-1}
{-m_{j,n}+1,\,\{\alpha_i+\beta_j+n_{i,n}+1\}_{i=1}^p,\,
 \{\beta_j-\beta_h-m_{h,n}+1\}_{h\ne j}}
{\{\alpha_i+\beta_j+1\}_{i=1}^p,\,
 \{\beta_j-\beta_h+1\}_{h\ne j}}{x}.
\end{multline*}
If \(m_{j,n}=0\), then \(B_n^{(j)}\equiv0\) by convention.
\end{coro}
Expanding these components yields the following coefficient formula.
\begin{coro}
For \(m_{j,n}\ge1\),
\(
B_{n}^{(j)}(x) = \sum_{k=0}^{m_{j,n}-1} B_{n}^{(j)}[k] x^k
\)
where
\begin{multline}
\label{PinheiroCoefiStepBn}
B_{n}^{(j)}[k] =
 \frac{\prod_{j=1, j \ne r(n,q)+1}^q (\beta_j- \beta_{r(n,q)+1} - m_{r(n,q)+1,n}+1)_{m_{j,n}} (m_{r(n,q)+1,n}-1)! (-1)^{m_{r(n,q)+1,n}-1}}{\prod_{i=1}^p (\beta_{r(n,q)+1}+m_{r(n,q)+1,n} + \alpha_i)_{n_{i,n}}}\\
	 \times \frac{\prod_{i=1}^p (\alpha_i+k+\beta_j+1)_{n_{i,n}}}{\prod_{i=1, i\ne j}^q (\beta_i-\beta_j-k)_{m_{i,n}} k! (m_{j,n}-k-1)! (-1)^{k}} .
\end{multline}
\end{coro}
The normalized type-II step-line form then has the following mixed moments.
\begin{coro}
For \(s\in\{1,\ldots,p\}\), these moments are
\begin{multline}
 \label{eq:normBn}
 \int_0^1 x^{k} x^{\alpha_s} B_{n} (x)\,\d x
 =\frac{\prod_{j=1, j \ne r(n,q)+1}^q (\beta_j- \beta_{r(n,q)+1} - m_{r(n,q)+1,n}+1)_{m_{j,n}}
}{\prod_{i=1}^p (\beta_{r(n,q)+1}+m_{r(n,q)+1,n} + \alpha_i)_{n_{i,n}}} \\ \times
 \frac{(m_{r(n,q)+1,n}-1)! (-1)^{ m_{r(n,q)+1,n} -1}
 \prod_{i=1}^p (\alpha_i-\alpha_{s} -k )_{n_{i,n}}}{\prod_{j=1}^q (\beta_j+ \alpha_{s} +k +1)_{m_{j,n}}},\qquad k\in\mathbb N_0.
 \end{multline}
 where \(n_{i,n}\) and \(m_{j,n}\) are given by \eqref{eq:overlineniB}.
\end{coro}
Next, consider the basic Christoffel transformations described in
\cite{banda_bidiagonal}. Define the affine cyclic shifts
\(\mathcal C_p\) and \(\mathcal C_q\) by
\begin{align*}
\mathcal C_p \vec{\alpha} = \begin{bNiceMatrix}
		\alpha_2 & \Cdots & \alpha_p & \alpha_1+1
\end{bNiceMatrix}, \quad
\mathcal C_q \vec{\beta} = \begin{bNiceMatrix}
		\beta_2 & \Cdots & \beta_q & \beta_1+1
\end{bNiceMatrix}.
 \end{align*}
For admissible parameter vectors, put
\begin{align*}
\vec{\alpha}^{(k)} \coloneq \mathcal C_p^k \vec{\alpha}, &&
\vec{\beta}^{(k)} \coloneq \mathcal C_q^k \vec{\beta}.
 \end{align*}
Using the affine extensions
\begin{align}
\label{eq:ext_nova}
\alpha_{sp+i}=\alpha_i+s && \text{for} && 1\le i\le p &&
\text{and} &&
\beta_{sq+j}=\beta_j+s && \text{for} && 1\le j\le q ,
\end{align}
this is equivalently
 \begin{align*}
\vec{\alpha}^{(k)} =(\alpha_{1+k}, \alpha_{2+k}, \ldots, \alpha_{p+k}), &&
\vec{\beta}^{(k)} = (\beta_{1+k}, \beta_{2+k}, \ldots, \beta_{q+k}),
 \end{align*}
and the transformed parameters remain in the admissible AT range. Indeed,
with \(u=(x^{\alpha_1},\ldots,x^{\alpha_p})\) and~\(v=(x^{\beta_1},\ldots,x^{\beta_q})\),
\begin{align*}
 u\mathfrak X_{[p]}^{\top}
 &= (x^{\alpha_2},\ldots,x^{\alpha_p},x^{\alpha_1+1}),&
 \mathfrak X_{[q]}v^{\top}
 &= (x^{\beta_2},\ldots,x^{\beta_q},x^{\beta_1+1})^{\top}.
\end{align*}
Consequently, the Christoffel-transformed matrices of measures remain mixed
Pi\~neiro matrices and satisfy
\begin{align*}
\left( \d \mu_{\vec{\alpha}, \vec{\beta}} \right)_L^{(k)} = \d \mu_{\vec{\alpha}^{(k)}, \vec{\beta}}, &&
\left( \d \mu_{\vec{\alpha}, \vec{\beta}} \right)_R^{(k)} = \d \mu_{\vec{\alpha}, \vec{\beta}^{(k)}}.
 \end{align*}
For brevity, write
\(A_n^{(i)}=(A_{\vec\alpha,\vec\beta})_n^{(i)}\). Then
\(\left(A^{(k)}_{L,\vec\alpha,\vec\beta}\right)^{(i)}_n
=\left(A_{\vec\alpha^{(k)},\vec\beta}\right)^{(i)}_n\) and
\(\left(A^{(k)}_{R,\vec\alpha,\vec\beta}\right)^{(i)}_n
=\left(A_{\vec\alpha,\vec\beta^{(k)}}\right)^{(i)}_n\).

The cyclic left Christoffel shift can now be substituted directly into the
normalized type-I formula.
\begin{coro}
 For \(n_{i,n}\ge1\), the component
 $\left( A_{\vec{\alpha}^{(k)}, \vec{\beta}} \right)^{(i)}_n$ has the following explicit representation with normalization \eqref{eq:normalization}:
 \begin{multline*}
\left( A_{\vec{\alpha}^{(k)}, \vec{\beta}} \right)^{(i)}_n(x)= \frac{\prod_{j=1}^p (\alpha_{j+k}+ \beta_{r(n,q)+1} + m_{r(n,q)+1,n}+1)_{n_{j,n}}}{\prod_{j=1}^q (\beta_j-\beta_{r(n,q)+1}-m_{r(n,q)+1,n})_{m_{j,n}}} \frac{\prod_{j=1}^q (\alpha_{i+k}+\beta_j+1)_{m_{j,n}}}{(n_{i,n}-1)!\prod_{j=1,j\ne i}^p (\alpha_{j+k}-\alpha_{i+k})_{n_{j,n}}} \\
\times
\pFq{p+q}{p+q-1}{-n_{i,n}+1,\,\{\alpha_{i+k}+\beta_s+m_{s,n}+1\}_{s=1}^q,\,\{\alpha_{i+k}-\alpha_{s+k}-n_{s,n}+1\}_{s=1,s\neq i}^p}{\{\alpha_{i+k}+\beta_s+1\}_{s=1}^q,\,\{\alpha_{i+k}-\alpha_{s+k}+1\}_{s=1,s\neq i}^p}{x}
\end{multline*}
where $n_{i,n}$ and $m_{j,n}$ are given by \eqref{eq:PinheiroA-ni}.
\end{coro}

Its leading coefficient follows by taking the highest-degree term.
 \begin{coro}
 For \(n_{i,n}\ge1\),
 \begin{align*}
\operatorname{LC} \left(A_{\vec{\alpha}^{(k)}, \vec{\beta}} \right)^{(i)}_n= \frac{\prod_{j=1}^p (\alpha_{j+k}+ \beta_{r(n,q)+1} + m_{r(n,q)+1,n}+1)_{n_{j,n}}}{\prod_{j=1}^q (\beta_j-\beta_{r(n,q)+1}-m_{r(n,q)+1,n})_{m_{j,n}}} \frac{(-1)^{n_{i,n}-1}\prod_{j=1}^q (\alpha_{i+k}+n_{i,n}+\beta_j)_{m_{j,n}}}{\prod_{j=1, j\ne i}^p (\alpha_{j+k}-\alpha_{i+k}-n_{i,n}+1)_{n_{j,n}} (n_{i,n}-1)! } ,
\end{align*}
where $n_{i,n}$ and $m_{j,n}$ are given by \eqref{eq:PinheiroA-ni}.
 \end{coro}

The right Christoffel shift is obtained similarly by cycling the
\(\beta\)-parameters.
 \begin{coro}
 For \(n_{i,n}\ge1\), the transformed component has the explicit representation
 \(\left(A_{\vec\alpha,\vec\beta^{(k)}}\right)^{(i)}_n\). If
 \(s=r(n,q)+1\), its normalization is
 \begin{align}
 \label{eq:normalization-right-transformed}
  & \int_0^1x^{m_{s,n}}x^{\beta_{s+k}}
 \left(A_{\vec\alpha,\vec\beta^{(k)}}\right)_n(x)\,\d x=1. \\
 &
\begin{multlined}[t][.95\textwidth]
\left( A_{\vec{\alpha}, \vec{\beta}^{(k)} } \right)^{(i)}_n(x)= \frac{\prod_{j=1}^p (\alpha_j+ \beta_{r(n,q)+1+k} + m_{r(n,q)+1,n}+1)_{n_{j,n}} }{ \prod_{j=1}^q (\beta_{j+k}-\beta_{r(n,q)+1+k}-m_{r(n,q)+1,n} )_{m_{j,n}} } \frac{\prod_{j=1}^q (\alpha_i+\beta_{j+k}+1)_{m_{j,n}} }{ (n_{i,n}-1)! \prod_{j=1,j\ne i}^p (\alpha_j-\alpha_i)_{n_{j,n}} } \\
\times
 \pFq{p+q}{p+q-1}{-n_{i,n}+1,\,\{\alpha_i+\beta_{s+k}+m_{s,n}+1\}_{s=1}^q,\,\{\alpha_i-\alpha_s-n_{s,n}+1\}_{s=1,s\neq i}^p}{\{\alpha_i+\beta_{s+k}+1\}_{s=1}^q,\,\{\alpha_i-\alpha_s+1\}_{s=1,s\neq i}^p}{x} ,
\end{multlined} 
\end{align}
where $n_{i,n}$ and $m_{j,n}$ are given by \eqref{eq:PinheiroA-ni}.
\end{coro}

Taking the highest-degree term gives the leading coefficient of the right
Christoffel transform.
 \begin{coro}
 For \(n_{i,n}\ge1\),
 \begin{align*}
\operatorname{LC} \left(\left(A_{\vec{\alpha}, \vec{\beta}^{(k)}} \right)^{(i)}_n\right)= \frac{\prod_{j=1}^p (\alpha_{j}+ \beta_{r(n,q)+1 +k} + m_{r(n,q)+1,n}+1)_{n_{j,n}}}{\prod_{j=1}^q (\beta_{j+k}-\beta_{r(n,q)+1+k}-m_{r(n,q)+1,n})_{m_{j,n}}} \frac{(-1)^{n_{i,n}-1}\prod_{j=1}^q (\alpha_{i}+n_{i,n}+\beta_{j+k})_{m_{j,n}}}{\prod_{j=1, j\ne i}^p (\alpha_{j}-\alpha_{i}-n_{i,n}+1)_{n_{j,n}} (n_{i,n}-1)! } ,
\end{align*}
where $n_{i,n}$ and $m_{j,n}$ are given by \eqref{eq:PinheiroA-ni}.
\end{coro}

Recalling Corollary~\ref{coro}, the coefficients in the bidiagonal
factorization are given by
\begin{align*}
	(L_k)_{n+1,n} &= \frac{\operatorname{LC} \left( \left( A_{\vec{\alpha}^{(k)}, \vec{\beta}} \right)^{(r(n,p)+1)}_n \right)}{\operatorname{LC} \left( \left( A_{\vec{\alpha}^{(k-1)}, \vec{\beta}} \right)^{(r(n+1,p)+1)}_{n+1} \right)},
&
	(U_k)_{n,n} &= \frac{\operatorname{LC} \left( \left( A_{\vec{\alpha}, \vec{\beta}^{(k-1)}} \right)^{(r(n,p)+1)}_n \right)}{\operatorname{LC} \left( \left( A_{\vec{\alpha}, \vec{\beta}^{(k)}} \right)^{(r(n,p)+1)}_{n} \right)}.
 \end{align*}
These identities are exact consequences of Corollary~\ref{coro}: inserting
the displayed leading coefficients gives the bidiagonal entries, and the
product \(L_1\cdots L_pU_q\cdots U_1\) is the recurrence matrix.

Substituting the contour-normalized polynomial coefficients
\eqref{eq:kcoeficientAn} and \eqref{eq:kcoeficientBn}, together with the
moments~\eqref{eq:momentosgeneralizadosAn} and
\eqref{eq:momentosgeneralizadosBn}, into \eqref{MixedRecurrence} yields the
following result.

In the next theorem we use the same contour
normalization in \eqref{eq:contourPinheiro} and
\eqref{eq:contourPinheiroB}: the unspecified common multiplicative constant
in each of those two representations is set equal to \(1\). This convention
is essential, because rescaling the type-II forms changes the off-diagonal
recurrence coefficients by the corresponding ratios of scaling constants.

With this convention fixed, the recurrence coefficients can be stated
explicitly.
\begin{teo}
	\label{teo:PiñeiroRecurrence}
	The Piñeiro recurrence coefficients \eqref{MixedRecurrence} that appear in the recurrence relations \eqref{MixedRII} are:
	\begin{align}
		\label{PiñeiroRecurrence}
		b^j_{\vec{n},\vec{m}}=&
\begin{multlined}[t][.825\textwidth]
\prod_{i=1}^p\dfrac{\left(\alpha_{i}+\beta_{\pi_t(j)}+m_{\pi_t(j)}+1\right)_{\left(\vec{n}+\vec{\mathfrak s}_{j+1}\right)_i}}{\left(\alpha_{i}+\beta_{\pi_t(j)}+m_{\pi_t(j)}+1\right)_{\left(\vec{n}+\vec{\mathfrak s}_{j}\right)_i}}\\
		\times\sum_{k=j}^q
		\prod_{i=1}^p\dfrac{\left(\alpha_{i}+\beta_{\pi_t(k)}+m_{\pi_t(k)}\right)_{n_i}}{\left(\alpha_{i}+\beta_{\pi_t(k)}+m_{\pi_t(k)}+1\right)_{\left(\vec{n}+\vec{\mathfrak s}_{j+1}\right)_i}}\dfrac{\prod_{i=1}^q
			\beta_{\pi_t(i)}-\beta_{\pi_t(k)}-m_{\pi_t(k)}}{\prod_{i=j,i\neq k}^q
			\beta_{\pi_t(i)}-\beta_{\pi_t(k)}-m_{\pi_t(k)}+m_{\pi_t(i)}},
\end{multlined}
		\\[4pt]
		b^0_{\vec{n},\vec{m}}=&
\begin{multlined}[t][.825\textwidth]
\prod_{j=1}^q\dfrac{\alpha_{\pi_\mathfrak{s}(1)}+\beta_j+n_{\pi_\mathfrak{s}(1)}+1}{\alpha_{\pi_\mathfrak{s}(1)}+\beta_j+n_{\pi_\mathfrak{s}(1)}+m_j+1}\,\prod_{i=1}^p\dfrac{
\alpha_{i}-\alpha_{\pi_\mathfrak{s}(1)}-n_{\pi_\mathfrak{s}(1)}-1}{
		\alpha_{i}-\alpha_{\pi_\mathfrak{s}(1)}-n_{\pi_\mathfrak{s}(1)}+n_{i}-1}\\
		+\sum_{k=1}^p
		\dfrac{\left(\alpha_{\pi_\mathfrak{s}(1)}+\beta_{\pi_\mathfrak{t}(1)}+n_{\pi_\mathfrak{s}(1)}+m_{\pi_\mathfrak{t}(1)}\right)}{\left(\alpha_{\pi_\mathfrak{s}(k)}+\beta_{\pi_\mathfrak{t}(1)}+n_{\pi_\mathfrak{s}(k)}+m_{\pi_\mathfrak{t}(1)}-1\right)\left(\alpha_{\pi_\mathfrak{s}(1)}-\alpha_{\pi_\mathfrak{s}(k)}-n_{\pi_\mathfrak{s}(k)}+n_{\pi_\mathfrak{s}(1)}+1\right)}\\
		\times\prod_{j=1}^q\dfrac{\alpha_{\pi_\mathfrak{s}(k)}+\beta_j+n_{\pi_\mathfrak{s}(k)}}{\alpha_{\pi_\mathfrak{s}(k)}+\beta_j+n_{\pi_\mathfrak{s}(k)}+m_j}\dfrac{\prod_{i=1}^p
			\alpha_{i}-\alpha_{\pi_\mathfrak{s}(k)}-n_{\pi_\mathfrak{s}(k)}}{\prod_{i=1,i\neq k}^p
			\alpha_{\pi_\mathfrak{s}(i)}-\alpha_{\pi_\mathfrak{s}(k)}-n_{\pi_\mathfrak{s}(k)}+n_{\pi_\mathfrak{s}(i)}},
\end{multlined}
		\\[4pt]
		b^{-i}_{\vec{n},\vec{m}}=&
\begin{multlined}[t][.825\textwidth]
\prod_{j=1}^q\dfrac{\left(\alpha_{\pi_s(i)}+\beta_j+n_{\pi_s(i)}\right)_{\left(\vec{m}-\vec{\mathfrak t}_i\right)_j}}{\left(\alpha_{\pi_s(i)}+\beta_j+n_{\pi_s(i)}\right)_{\left(\vec{m}-\vec{\mathfrak t}_{i+1}\right)_j}} \sum_{k=i}^p
		\prod_{j=1}^q\dfrac{\left(\alpha_{\pi_s(k)}+\beta_j+n_{\pi_s(k)}\right)_{\left(\vec{m}-\vec{\mathfrak t}_{i+1}\right)_j}}{\left(\alpha_{\pi_s(k)}+\beta_j+n_{\pi_s(k)}+1\right)_{m_j}} \\
	\times\dfrac{\prod_{j=1}^p
		\alpha_{\pi_s(j)}-\alpha_{\pi_s(k)}-n_{\pi_s(k)}}{\prod_{j=i,j\neq k}^p
		\alpha_{\pi_s(j)}-\alpha_{\pi_s(k)}-n_{\pi_s(k)}+n_{\pi_s(j)}},
\end{multlined}
\end{align}
for \(j\in\{1,\ldots,q\}\) and \(i\in\{1,\ldots,p\}\).
The formula for \(b^{-i}_{\vec n,\vec m}\) is asserted when the two shifted
multi-indices in that term are nonnegative; otherwise that backward term is
absent.
\end{teo}

\begin{proof}
The result is obtained by exact substitution of the polynomial coefficients and moments in \eqref{eq:kcoeficientBn}, \eqref{eq:momentosgeneralizadosAn}, \eqref{eq:kcoeficientAn}, and \eqref{eq:momentosgeneralizadosBn}
 in the general coefficient
formulas \eqref{ReMixedRecurrencej}, \eqref{ReMixedRecurrence0} and \eqref{ReMixedRecurrencei}.
\end{proof}
The cyclic notation in~\eqref{eq:ext_nova}
gives a compact form of the preceding step-line
recurrence coefficients.
Fix~\(n\ge0\) and write
\begin{align*}
 n=p\mathfrak r+\mathfrak l=qr+l, &&
 0\le\mathfrak l<p, && 0\le l<q.
\end{align*}
Let \(\vec n\) be the \(p\)-component step-line multi-index of modulus \(n\),
and let \(\vec m^{\,A}\) be the \(q\)-component step-line multi-index of
modulus \(n\). The type-II multi-index paired with \(\vec n\) is
\(\vec m^{\,B}=\vec m^{\,A}+\vec e_{l+1}\), of modulus \(n+1\).
The~relevant shifted components are
\begin{align*}
	\left(\vec{n}+\vec{\mathfrak s}_j\right)_i
	& =\left\lceil \frac{|\vec{n}|+j+1-i}{p}\right\rceil=\mathfrak r+\left\lceil \frac{\mathfrak l+j+1-i}{p}\right\rceil,& j\in\{0,1,\ldots,q\},\\
	\left(\vec m^{\,B}-\vec{\mathfrak t}_i\right)_j
	&=\left\lceil \frac{|\vec m^{\,A}|-i+2-j}{q}\right\rceil
	=r+\left\lceil \frac{l-i+2-j}{q}\right\rceil,
	& i\in\{0,1,\ldots,p\}.
\end{align*}

To keep the two normalizations distinct, let \(\widehat B_N\) denote the
step-line type-II form obtained from \eqref{eq:contourPinheiroB} with
multiplicative constant \(1\). Put
\begin{align}
\label{eq:raw-step-pivot}
 s_N:=r(N,q)+1,&& d_N:=\left\lfloor\frac Nq\right\rfloor,
 &&
 \widehat\kappa_N:=\widehat B_N^{(s_N)}[d_N].
\end{align}
By \eqref{eq:kcoeficientBn},
\begin{align}
\label{eq:raw-step-pivot-explicit}
 \widehat\kappa_N=
 \frac{\prod_{a=1}^p
 (\alpha_a+\beta_{s_N}+d_N+1)_{n_{a,N}}}
 {
 \prod_{h\ne s_N}(\beta_h-\beta_{s_N}-d_N)_{m_{h,N}}
 d_N!\,(m_{s_N,N}-d_N-1)!\,(-1)^{d_N}}.
\end{align}
Here \(m_{s_N,N}=d_N+1\), so the last factorial in the denominator is
\((m_{s_N,N}-d_N-1)!=0!=1\).
The monic form used in \eqref{eq:normalizationB} is therefore
\begin{align}
\label{eq:raw-to-monic-step}
 B_N=\widehat B_N/\widehat\kappa_N.
\end{align}

Specializing the preceding recurrence formulas to the step line gives the
monic coefficients.
\begin{coro}
	\label{coro:PiñeiroRecurrenceSL}
	With \(n=p\mathfrak r+\mathfrak l=qr+l\) as above, set
	\(b_n^{-i}=0\) whenever \(i>n\). For
	\(j\in\{1,\ldots,q\}\) and
	\(1\le i\le\min(p,n)\), the remaining Piñeiro recurrence coefficients on
	the step line are
\begin{align*}
		b^j_n
		&=
		\begin{multlined}[t][.875\textwidth]
		\dfrac{\widehat\kappa_{n+j}}{\widehat\kappa_n}
		\prod_{i=1}^p
		\dfrac{\left(\alpha_i+\beta_{{l+1}+j}+r+1+\mathfrak r\right)_{
		\left\lceil\frac{\mathfrak l+j+2-i}{p}\right\rceil}}
		{\left(\alpha_i+\beta_{{l+1}+j}+r+1+\mathfrak r\right)_{
		\left\lceil\frac{\mathfrak l+j+1-i}{p}\right\rceil}}
		\\
		{}\times
		\sum_{k=j}^q
		\prod_{i=1}^{\mathfrak l}
		\left(\alpha_i+\beta_{k+{l+1}}+r\right)\prod_{i=1}^p
		\dfrac{\left(\alpha_{i+\mathfrak l}+\beta_{k+{l+1}}+r\right)_{\mathfrak r}}
		{\left(\alpha_i+\beta_{k+{l+1}}+r+1\right)_{
		\mathfrak r+\left\lceil\frac{\mathfrak l+j+2-i}{p}\right\rceil}}
		\dfrac{
		\prod_{i=1}^q
		\left(\beta_i-\beta_{k+{l+1}}-r\right)}
		{
		\prod_{
		i=j ,
		i\neq k 
		}^q
		\left(\beta_{i+{l+1}}-\beta_{k+{l+1}}\right)},
		\end{multlined}
\\[4pt]
		b^0_n
		&=
		\begin{multlined}[t][.875\textwidth]
		\prod_{j=1}^q
		\dfrac{\alpha_{\mathfrak l+1}+\beta_j+\mathfrak r+1}
		{\alpha_{\mathfrak l+1}+\beta_{{l+1}+j}+r+\mathfrak r+1}
		\prod_{i=1}^p
		\dfrac{\alpha_i-\alpha_{\mathfrak l+1}-\mathfrak r-1}
		{\alpha_{\mathfrak l+i}-\alpha_{\mathfrak l+1}-1}
		\\
		+\sum_{k=1}^p
		\dfrac{\alpha_{\mathfrak l+1}+\beta_{{l+1}}+\mathfrak r+r+1}
		{\left(\alpha_{\mathfrak l+k}+\beta_{{l+1}}+\mathfrak r+r\right)
		\left(\alpha_{\mathfrak l+1}-\alpha_{\mathfrak l+k}+1\right)}
		\prod_{j=1}^q
		\dfrac{\alpha_{\mathfrak l+k}+\beta_j+\mathfrak r}
		{\alpha_{\mathfrak l+k}+\beta_{{l+1}+j}+\mathfrak r+r}
		\dfrac{
		\prod_{i=1}^p
		\left(\alpha_i-\alpha_{\mathfrak l+k}-\mathfrak r\right)}
		{
		\prod_{
		{i=1 ,
		i\neq k}}^p
		\left(\alpha_{\mathfrak l+i}-\alpha_{\mathfrak l+k}\right)},
		\end{multlined}
		\\[4pt]
		b^{-i}_n
		&=
		\begin{multlined}[t][.875\textwidth]
		\dfrac{\widehat\kappa_{n-i}}{\widehat\kappa_n}
		\prod_{j=1}^q
		\dfrac{\left(\alpha_{p+\mathfrak l-i+1}+\beta_j+\mathfrak r+r\right)_{
		\left\lceil\frac{l+2-i-j}{q}\right\rceil}}
		{\left(\alpha_{p+\mathfrak l-i+1}+\beta_j+\mathfrak r+r\right)_{
		\left\lceil\frac{l+1-i-j}{q}\right\rceil}}
		\\
		{}\times
		\sum_{k=i}^p
		\prod_{j=1}^{l+1}
		\dfrac{1}{\alpha_{p+\mathfrak l-k+1}+\beta_j+\mathfrak r+1}
		\dfrac{
		\prod_{j=1}^q
		\dfrac{\left(\alpha_{p+\mathfrak l-k+1}+\beta_j+\mathfrak r\right)_{
		r+\left\lceil\frac{l+1-i-j}{q}\right\rceil}}
		{\left(\alpha_{p+\mathfrak l-k+1}+\beta_{{l+1}+j}+\mathfrak r+1\right)_r}
		\prod_{j=1}^p
		\left(\alpha_j-\alpha_{p+\mathfrak l-k+1}-\mathfrak r\right)}
		{
		\prod_{
		{j=i ,
		j\neq k}}^p
		\left(\alpha_{p+\mathfrak l-j+1}-\alpha_{p+\mathfrak l-k+1}\right)}.
		\end{multlined}
\end{align*}
Here \((z)_a\) for a negative integer \(a\) has the meaning specified in
\eqref{eq:negative-pochhammer}.
\end{coro}
\begin{proof}
We substitute the coefficients formulas obtained in  Theorem~\ref{teo:PiñeiroRecurrence}, taking into account \eqref{eq:raw-step-pivot}, and~\eqref{eq:ext_nova}. This gives the three
coefficients for the contour-normalized forms \(\widehat B_N\), namely the
displayed expressions before the pivot ratios are inserted.
Equation~\eqref{eq:raw-to-monic-step} then multiplies every forward
coefficient by \(\widehat\kappa_{n+j}/\widehat\kappa_n\), leaves the central
coefficient unchanged, and multiplies every backward coefficient by
\(\widehat\kappa_{n-i}/\widehat\kappa_n\). 

\end{proof}

\section{Conclusion}

We have related the step-line mixed multiple orthogonality problem to a
banded recurrence matrix and, under the stated Gauss--Borel assumptions, to a
product of lower and upper bidiagonal factors. For the mixed Pi\~neiro system,
the contour representations yield terminating hypergeometric formulas and
explicit recurrence and factorization coefficients. The distinction between
the contour-normalized forms \(\widehat B_n\) and the monic forms \(B_n\) is
essential: their pivot ratios give the off-diagonal coefficients in the monic
step-line recurrence. The bidiagonal factorization is then obtained exactly
from the Christoffel leading-coefficient formulas; computational checks may
verify the resulting expressions, but they are not used as a proof.

For the mixed Pi\~neiro system, it remains a topic for future research to
exploit the bidiagonal factorization in order to derive sufficient conditions
on the parameters defining the system that guarantee the nonnegativity of the
associated recurrence matrix. Such a characterization is expected to have
applications to the numerically stable computation of its eigenvalues and
eigenvectors, and consequently to the construction of quadrature formulas
associated with mixed multiple orthogonality.

\section*{Acknowledgements}
The authors thanks Thomas Wolfs for enlightening discussions in the beginning of the investigation presented here.

\end{document}